\documentclass[12pt,letterpaper,titlepage]{amsart}
\usepackage{amsmath, amssymb, amsthm, amsfonts,amscd,xr}
\usepackage[all]{xy}
\CompileMatrices
\input epsf.tex
\newcommand{\N}{\mathbb{N}}

\newcommand{\R}{\mathbb{R}}
\newcommand{\C}{\mathbb{C}}

\def\beq{\begin{equation}}
\def\eeq{\end{equation}}

\def\arr{\hbox to 20pt{\rightarrowfill}}

\def\hat{\widehat}
\def\1{\mathbf 1}
\def\e{\varepsilon}
\def\tilde{\widetilde}
\def\sl{\mathfrak {sl}}

\def\Gl{\mathrm {Gl}}
\def\Sl{\mathrm {Sl}} 
\def\SO{\mathrm {SO}} 
 
\def\max{\mathrm {max}}
\def\A{\mathcal A} 
\def\D{\mathcal D} 
\def\cC{\mathcal C} 
\def\H{\mathcal H} 
\def\HC{\mathcal HC}
\def\O{\mathcal O} 
\def\M{\mathcal M} 
\def\W{\mathcal W} 
 
\def\cR{\mathcal R} 

\def\cZ{\mathcal Z}
 
\def\U{\mathcal U} 
\def\F{\mathcal F}

\def\nin{\noindent}

\newenvironment{res} 
               {\begin{equation} 
\begin{minipage}{0.85\textwidth}} 
               { \end{minipage}\end{equation} } 
\def\ber{\begin{res} } 
\def\eer{\end{res}} 
 
\numberwithin{equation}{section} 
\newtheorem{thm}{Theorem}[section]

\newcommand{\oline}{\overline}

\newtheorem{lemma}[thm]{Lemma} 
\newtheorem{lem}[thm]{Lemma} 
 
\newtheorem{cor}[thm]{Corollary} 
\newtheorem{ex}[thm]{Example} 
\newtheorem{prop}[thm]{Proposition} 
 
\newtheorem{dfn}[thm]{Definition} 
\newtheorem{rem}[thm]{Remark} 
\newtheorem{problem}[thm]{Problem}

\makeatletter 
\def\section{\@startsection {section}{1}{\z@}{3.5ex plus 1ex minus 
    .2ex}{2.3ex plus .2ex}{\large\bf}} 
    \def\subsection{\@startsection{subsection}{2}{\z@}{3.25ex plus 1ex minus 
 .2ex}{1.5ex plus .2ex}{\bf}} 
\makeatother 
\def\qed{\hfill $\square$\par\vspace{5pt}} 
\def\bysame{\leavevmode\hbox to3em{\hrulefill}\,} 

\def\Ad{\operatorname{Ad}} 
\def\supp{\operatorname{supp}} 
\def\spec{\operatorname{spec}}
\def\e{\epsilon}

\def\af{\mathfrak{a}}

\def\End{\operatorname {End}}
\def\gf{\mathfrak{g}} 
 
\def\hf{\mathfrak{h}}

\def\Ind{\operatorname{Ind}}
\def\id{\operatorname{id}}
\def\Irr{\operatorname{Irr}}
\def\kf{\mathfrak{k}} 
 
\def\m{\mathfrak{m}} 
\def\nf{\mathfrak{n}} 
 
\def\pf{\mathfrak{p}}

\def\tf{\mathfrak{t}}

\def\tr{\operatorname{tr}} 

\def\A{\mathcal{A}}

\def\Ad{\operatorname{Ad}}
\def\ad{\operatorname{ad}}
\def\O{\mathcal{O}} 
\def\Re{\operatorname{Re}}
\def\Hom{\operatorname{Hom}}

\def\im{\operatorname{im}}
\def\pr{\operatorname{pr}}
\def\rank{\operatorname{rank}}

\def\Res{\operatorname{Res}}

\def\SAF{\mathcal{SAF}}

\def\Cc{\mathcal{C}} 
 
\def\M{\mathcal{M}} 
\def\I{\mathcal{I}} 
\def\A{\mathcal{A}} 
\def\H{\mathcal{H}} 
\def\bH{\mathbf{H}} 
\def\F{\mathcal{F}}

\def\E{\mathcal{E}} 
 
\def\S{\mathcal{S}} 
 
\def\bV{\mathbf{V}}

\def\W{\mathcal{W}}

\def\bs{\backslash}
\def\ra{\rangle}
\def\la{\langle}

\makeindex 
\title[Smooth globalizations]
{Smooth Fr\'echet globalizations of Harish-Chandra modules}
\begin{document}
\author{Joseph Bernstein}
\author{Bernhard Kr\"otz}

\address{School of Mathematical Sciences\\ Tel Aviv University\\
Ramat Aviv\\Tel Aviv 69978\\ Israel\\ bernstei@post.tau.ac.il}
\address{Institut f\"ur Mathematik\\ Universit\"at Paderborn\\
Warburger Stra\ss{}e 100\\ 33098 Paderborn\\ Germany\\bkroetz@math.uni-paderborn.de} 

\date{\today} 
\thanks{}
\maketitle

\section{Introduction}

Let $G$ be a linear reductive real Lie group  $G$ with Lie algebra 
$\gf$\footnote{Throughout this text Lie groups will be denoted 
by upper case Latin letters, $G$, $K$, $N$ ... , and their 
corresponding Lie algebras by lower case German letters
$\gf$, $\kf$, $\nf$ etc.}. Let us fix a maximal compact subgroup $K$ of $G$. 
The representation theory of $G$ admits an algebraic underpinning 
encoded in the notion of a {\it Harish-Chandra module}. 
\par By a Harish-Chandra module we shall understand a finitely generated 
$(\gf, K)$-module with finite $K$-multiplicities. 
Let us denote by $\HC$ the category whose objects are Harish-Chandra
modules and whose morphisms are linear $(\gf, K)$-maps. 
By a {\it globalization} of a Harish-Chandra module $V$ 
we understand a representation $(\pi, E)$ of $G$ such that 
the $K$-finite vectors of $E$ are isomorphic to $V$ as a $(\gf, K)$-module. 
\par Let us denote by $\SAF$ the category whose objects are 
smooth admissible moderate growth Fr\'echet representations of $G$ 
with continuous linear $G$-maps as morphisms. We 
consider the functor:

$$\F: \SAF \to \HC, \ \ E\mapsto E^{K-{\rm fin}}:=\{K-\hbox{finite vectors of E}\}\, .$$

\par The Casselman-Wallach globalization theorem (\cite{Cas}, \cite{W0} and  \cite{W2}, Sect. 11) 
essentially asserts that $\F$ is an equivalence of categories.  To phrase it differently, each  
Harish-Chandra module $V$  admits an $\SAF$-globalization $(\pi, V^\infty)$
which is unique up to isomorphism.  
It follows that 
$$V^\infty=\pi(\S(G))V$$
where $\S(G)$ is the Schwartz-algebra of rapidly decreasing functions on $G$, 
and $\pi(\S(G))V$ stands for the vector space spanned by $\pi(f)v$ for $f\in \S(G)$, 
$v\in V$. In particular, $V$ is irreducible if and only if $V^\infty$ is an 
algebraically simple $\S(G)$-module.

\par The objective of this paper is to present a new approach to the
globalization theorem.  
\par Our approach starts with 
a thorough study investigation of the topological nature of $\SAF$-globalizations.  
A norm $p$ on a Harish-Chandra module $V$  will be called {\it $G$-continuous}
provided the completion $V_p$ of the normed space $(V, p)$  gives rise to 
a Banach-representation of $G$. We introduce the {\it Sobolev order} on the 
set of $G$-continuous norms:
$$p\prec q :\iff (\exists C>0, k\in \N_0) (\forall v \in V) \qquad p(v) \leq C q_k (v)$$
where $q_k$ refers to the $k$-th Sobolev norm of $q$.  Two $G$-continuous norms
will be called {\it Sobolev-equivalent}  provided $p\prec q$ and $q\prec p$.
\par Casselman's subrepresentation theorem implies that every Harish-Chandra
module admits a $G$-continuous norm. Within our terminology the
Casselman-Wallach theorem now reads:

\begin{thm} Any two $G$-continuous norms on a Harish-Chandra module $V$ are 
Sobolev-equivalent. 
\end{thm} 

On a technical level it is quite cumbersome to deal with 
arbitrary $G$-continuous norms $p$. However, the algebraic fact that $K$-
multiplicities on a Harish-Chandra module are polynomially bounded, implies 
that the smooth vectors
$V_p^\infty$ form a nuclear Fr\'echet space. It implies that 
every $G$-continuous norm 
is Sobolev-equivalent to a $K$-invariant Hermitian norm (see Theorem \ref{th=nuc} below).
\par It is convenient to introduce an auxiliary notion and 
call a Harish-Chandra module {\it good} provided it admits a unique
$\SAF$-globalization. Note that 
the Casselman-Wallach Theorem is the assertion that all Harish-Chandra 
modules are good. Using elementary functional analysis we show in Section 7 that a Harish-Chandra module is 
good if and only if its associated matrix coefficients satisfy certain lower bounds which 
are uniform in the $K$-types (see Theorem \ref{thm=lobo}). 

\par Section 8  with Appendix A is devoted to minimal principal series
representations( i.e. representations which are induced off a minimal 
parabolic subgroup) and  their canonical 
Hilbert-globalizations as subspaces of $L^2(K)$.
For such representations  we define a Dirac-type 
sequence and establish essentially optimal uniform lower bounds for 
$K$-finite matrix coefficients (see Theorem \ref{th=1} below).  As a consequence 
we obtain that Harish-Chandra modules $V$ of the minimal principal series are good.
In addition we exhibit a generator $\xi\in V$ and an explicit linear
continuous  
section of the quotient homomorphism 
$$\S(G) \to V^\infty, \ \ f\mapsto \pi(f)\xi \, , $$
which depends holomorphically on the representation parameter of $\pi$ 
(see Theorem \ref{thm=main} and \ref{th=2}).

\par According to Casselman's subrepresentation 
theorem one can embed every Harish-Chandra module into a parabolically induced representation. In view of our results in Section 6 this implies that every Harish-Chandra 
module admits a minimal and maximal $G$-continuous norm with respect 
to the Sobolev order $\prec$ (see Cor. \ref{cor=mm}). Let us note that
this important technical 
step is likewise implied by Wallach's upper bounds on 
matrix coefficients (see \cite{W}, Theorem 4.3.5).

\par In Section 9 we show that "good" is preserved by extensions, induction and 
tensoring  with finite dimensional representations.  These elementary technical
results  are useful in the sequel. In particular 
it follows that the task to show 
that all Harish-Chandra modules are good is reduced to irreducible modules.  

\par In Section 10 we present Casselman's technique of holomorphic 
deformations of Harish-Chandra modules which, via Langlands classification,
leaves us to establish that all Harish-Chandra modules of the discrete series 
are good.  For a module $V$ of the discrete series we proceed as follows: we embed 
$V$ into a minimal principal series representation and are left to show  that the unitary norm 
on $V$ is Sobolev equivalent to the minimal norm. This in turn is reduced to 
a familiar result on meromorphic continuation of certain distributions
(see \cite{B} and Appendix B of this paper). 
In this context we  wish to point that the fact that discrete series modules 
are good can be also proved using Wallach's upper bounds  (\cite{W},
Theorem 4.3.5 and  \cite{W2}, Prop. 11.7.4).

\par In summary, we provide a functional analytic language for
globalizations and emphasize that our main input, i.e.  the estimates 
for minimal principal series representation in Section 8, cannot be deduced 
from the Casselman-Wallach theorem. 
In addition our bounds for matrix coefficients 
are locally uniform in the representation parameters 
and yield a Casselman-Wallach theorem for holomorphic 
families of Harish-Chandra modules (see Theorem \ref{th=2a}). 
This,  for instance,  is useful for the theory of Eisenstein series (see Theorem \ref{th=3} and Remark \ref{rem=Eis}).
    
\par It was our intention to write an essentially self-contained 
account on the subject which is accessible to graduate students. The reading
requires a good understanding 
of functional analysis and some  basic knowledge about real reductive groups
and Harish-Chandra modules as to be found in Wallach's text book \cite{W}, Sect. 1-3.

\section{Basic representation theory I: growth of representations}
We begin with a discussion of scale structures on Lie groups which give
us an appropriate notion of size on the group. We then collect 
a few standards definition and facts of representation theory
on topological vector spaces. 
After that we discuss growth issues of representations and introduce the 
notion of F-representation. We show that the category 
of smooth F-representations is isomorphic to a category 
of non-degenerate algebra representations of a
certain Schwartz algebra. 

\subsection{Scale structures on Lie groups}

Throughout this text $G$ will denote a Lie group. It is our objective to obtain a notion of size on $G$
which will be suitable to define growth for  a representation. 

\par By a {\it scale} on $G$  we understand a function $s:G\to \R^+$ such that: 
\begin{itemize}
\item $s$ and $s^{-1}$ are locally bounded, 
\item $s$ is {\it submultiplicative}, i.e. $s(gh)\leq s(g) s(h)$ for all $g, h\in G$. 
\end{itemize}

We introduce an ordering $\preccurlyeq$  on the space of  scale functions by 
$$s\preccurlyeq s' :\iff (\exists C>0, N\in \N)(\forall g\in G) \quad s(g) \leq C s'(g)^N\, .$$ 
This gives us a notion of equivalence $\sim$ on scale functions:
$$s\sim s' :\iff  s\preccurlyeq s' \quad\hbox{and}\quad s'\preccurlyeq s'\, .$$
By a {\it scale structure} on $G$ we understand an equivalence class $[s]$
of a scale function $s$. 

\par Note that every equivalence class $[s]$ admits a continuous representative. Henceforth we will only consider 
continuous scale functions. 
\par Let us discuss the various natural scale structures.

\subsubsection{The maximal scale structure}
Suppose that $G$ is connected and fix a left-invariant Riemannian 
metric ${\bf g}$ on $G$. Associated to ${\bf g}$ is 
the distance function $d(g,h)$, i.e. the minimum length of piecewise smooth curve 
joining $g$ and $h$ in $G$. Note that $d(\cdot, \cdot)$ is left $G$-invariant and hence can be recovered 
from 
$$d(g):=d(g, \1) \qquad (g\in G)\, .$$ 
Note that $d(g)$ is {\it subadditive}, i.e. $d(gh)\leq d(g) + d(h)$ for all $g, h\in G$.  
Hence $s_\max (g):=e^{d(g)}$ defines a scale function on $G$. 
This scale is maximal in the sense that for any other scale $s$ we have 
$s\preccurlyeq s_\max$ (see \cite{Ga}, Lemme 2).  In particular, the equivalence class $[s_\max]$ of $s_\max$ is 
independent of the choice of the particular left invariant metric. In the sequel we will refer 
to $[s_\max]$ as the {\it maximal scale structure}. 

\subsubsection{Algebraic scale structure} 
Let $G$ be a real algebraic group. We fix  a faithful algebraic representation $\iota: G\to \Gl(n,\R)$. 
Then 
$$\|g\|:= \tr (\iota(g)\iota(g)^t)+\tr (\iota(g^{-1})\iota(g)^{-t})\, \qquad (g\in G)$$ 
defines a smooth scale function on $G$. 
If we choose another faithful algebraic representation $\iota': G\to \Gl(n',\R)$, and if 
$\|\cdot\|'$ is the associated scale on $G$, then $\|\cdot\|$ and $\|\cdot\|'$
are equivalent. The resulting scale structure on $G$ will be referred to as the 
{\it algebraic scale structure}. 
We often refer to $\|\cdot\|$ as a {\it norm} on $G$ -- see \cite{W}, Sect. 2.A.2 for the notion of 
norm on a reductive group. 

\begin{lem} Let $G$ be a connected real reductive group. Then the algebraic and the maximal 
scale structure coincide. 
\end{lem}
\begin{proof} Let $\iota: G\to \Gl(n,\R)$ be a faithful representation and henceforth view 
$G$ as a subgroup of $\Gl(n,\R)$.  
We recall the Cartan decomposition $G=KAK$ where $K<G$ is a maximal compact 
sungroup and $A$ a non-compact torus. 
{}From the definitions of
the scale structures
involved it is easy to see that they 
coincide on $A$. The assertion follows. 
\end{proof}

\begin{rem} The maximal and algebraic scale structure of the algebraic group $G=(\R, +)$
are different (polynomial versus exponential growth).  
\end{rem}  

In the sequel we will understand in this paper by a Lie group $G$ a pair $(G, [s])$ where 
$[s]$ is a scale structure. If $G$ is real reductive, then $[s]$ shall be the maximal scale structure.

\subsection{Representations on topological vector spaces}

All topological vector spaces $E$ considered in this paper are 
understood to be locally convex. We denote by $\Gl(E)$ the group 
of all topological linear isomorphisms of $E$. 

\par Let $G$ be a Lie group and $E$ a topological vector space. 
By a {\it representation} $(\pi, E)$ of $G$ 
on $E$ we understand a homomorphism 
$\pi: G\to \Gl(E)$ such that the resulting  action 
$G\times E \to E $ is continuous. We emphasize that 
continuity is requested in both variables.
For an element 
$v\in E$ we shall denote by 
$$\gamma_v: G \to E, \ \ g\mapsto \pi(g)v$$
the corresponding continuous orbit map. 
The following Lemma is standard (cf. \cite{War}, Sect. 4.1).

\begin{lem}\label{repcrit} Let $G$ be a Lie group, $E$ a topological vector space, 
$\pi: G\to \Gl(E)$ a group homomorphism and $G\times E \to E$ 
the resulting action. Then the following 
statements are equivalent: 

\begin{enumerate} \item The action $G\times E\to E$ is continuous, i.e.
$(\pi, E)$ is a representation.  
\item \begin{enumerate} \item There exists a dense subset $E_0\subset E$ such
    that for all $v\in E_0$ the orbit map $\gamma_v: G\to E$
is continuous. 
\item For every compact subset $Q$ of $G$ the set $\{\pi(g)\mid g\in Q\}$ 
is an equicontinuous set of linear endomorphisms of $E$. 
\end{enumerate}
\end{enumerate}
\end{lem}

If $\pi: G\to \Gl(E)$ is a group homomorphism, then we say $\pi$ is  {\it locally 
equicontinuous} if condition (b) in the Lemma above is satisfied.

\begin{rem} (a)  Let $\pi: G\to \Gl(E)$ be homomorphism. If $E$ is a Banach space, 
then, in view of 
the uniform boundedness principle, the following statements are equivalent:
\begin{itemize}
\item The action $G\times E\to E$ is continuous, i.e. $(\pi, E)$ is a representation. 
\item For all $v\in E$ the orbit map $\gamma_v$ is continuous. 
\end{itemize}
In the existing literature one mostly considers representation on Banach spaces and uses 
the second bulleted item as a definition for representation. 
Let us emphasize that these two notions will be 
different in general. 
\par\nin (b) Suppose that $(\pi, E)$ is a representation on a
semi-normed space $E$. Then all 
operator norms  of $\pi(g)$ are locally bounded in $g\in G$.  
\end{rem}

\par If $(\pi, E)$ is a representation, then we call a continuous  semi-norm 
$p$ 
on $E$ a {\it $G$-continuous semi-norm}, if $G\times (E, p)\to (E, p)$ is 
continuous.  Here $(E, p)$ stands 
for the vector space $E$ endowed with the topology induced from 
the semi-norm $p$.

\begin{rem} \label{rem=111}Let $p$ be a $G$-continuous semi-norm on a 
representation module $E$ and $E_p$ be the completion 
of $(E, p)$. As $G\times (E,p)\to (E,p)$ is continuous, 
we obtain a representation of $G$ on the Banach space $E_p$. 
\end{rem}

Let $(\pi, E)$ be a representation of $G$. If $E$ is a Banach (Hilbertian, Fr\'echet)  space,   then we 
speak of a {\it Banach} ({\it Hilbertian}, {\it Fr\'echet}) {\it representation} of $G$. 

\subsection{Growth of a representation}

In this section $G=(G,[s])$ denotes  a Lie group with scale structure $[s]$.  

\par Let  $(\pi, E)$ be a representation of $G$ on a semi-normed space $(E, p)$. 
Then 
$$s_\pi: G\to \R^+, \ \ g\mapsto \|\pi(g)\|$$
is a scale.  We call $s_\pi$ the scale associated to $(\pi, E)$. 
We will say that $(\pi, E)$ is $[s]$-{\it bounded} provided $s_\pi  \preccurlyeq s$. 
\par A $G$-continuous semi-norm $p$ on a representation module $E$ will be 
called $[s]$-{\it bounded} provided  $G\times (E, p)\to (E, p)$ is a 
$[s]$-bounded representation.

\begin{dfn} A representation $(\pi, E)$ of $G=(G, [s])$ will be called an  {\it $F$-representation} provided 
$E$ is a Fr\'echet space whose topology 
is induced by  a countable family of $G$-continuous 
$[s]$-bounded semi-norms $(p^n)_{n\in\N}$.
\end{dfn}
Let us emphasize that Fr\'echet spaces are complete topological vector spaces.
In the context of $F$-representations this will play an important role
when it comes to vector valued integration.

\begin{ex} (a) If $[s]$ is the maximal scale structure, then 
any representation on a semi-normed space is $[s]$-bounded.  
\par\nin (b) Let $G=(\R,+)$ endowed with the algebraic scale structure. Then 
a character $\pi: G\to \C^*$ is $[s]$-bounded if and only if $\pi$ is unitary. 
\end{ex}

\begin{rem}{\rm (Fr\'echet representations versus $F$-representations)} Let us emphasize 
that a Fr\'echet representation is not necessarily an $F$-representation. Here are some 
examples: 
\par\nin (a) Let $G$ be non-compact connected Lie group and $E=C(G)$ be the space of continuous function 
on $G$. Then $E$, endowed with the topology of compact convergence, becomes a
Fr\'echet space. Let $\pi$ denote either the left or right regular action of $G$ on $E$. 
Then $(\pi, E)$ is a Fr\'echet but not an $F$-representation. 
\par\nin (b) Let $G=\Sl(2,\R)$, $B<G$ the standard Borel subgroup 
and $\chi: B\to \C^*$ a character. Let $E$ be the $G$-module 
of  hyperfunction 
sections of the line bundle $G\times_B  \C_\chi\to G/B$. 
As a topological vector space $E$ is isomorphic to the hyperfunctions 
on the circle, hence a Fr\'echet space. This yields a Fr\'echet representation 
which is not an $F$-representation. 
\par More generally,  if $(\pi, E)$ is the maximal globalization of a Harish-Chandra module 
(in the sense of Schmid, see \cite{WS}), then $(\pi, E)$ is a Fr\'echet representation but not an $F$-representation.   
\end{rem}

Recall that the category of Fr\'echet spaces is closed under taking 
closed subspaces and quotients by closed subspaces. The same holds for the
category of  F-representations.  We record this fact, but skip 
the very easy proof:

\begin{lem} \label{lem=222} Let $(\pi, E)$ be an $F$-representation
and $H\subset E$ a closed $G$-invariant subspace. Then the 
corresponding sub and  quotient representation on $H$, resp. $E/H$, are 
$F$-representations. 
\end{lem}

\subsubsection{Representations of moderate growth}
In \cite{Cas} Casselman calls a Fr\'echet representation $(\pi, E)$ of a real reductive group $G$ of  {\it moderate growth}
provided for any semi-norm $p$ on $E$ there exists a semi-norm $q$
on $E$ and an integer $N>0$ such that 
$$p(\pi(g)v)\leq  \|g\|^N  q(v)\qquad (g\in G)\, .$$
For an arbitrary Lie group $G=(G,[s])$ one thus might call 
a representation of {\it moderate growth} if for any semi-norm $p$ on $E$ there exists a semi-norm $q$
on $E$ and an integer $N>0$ such that 
$$p(\pi(g)v)\leq s(g)^N  q(v)\qquad (g\in G)\, .$$

\begin{lem} Let $(\pi, E)$ be a Fr\'echet representation of the Lie group 
$(G, [s])$. Then the following statements are equivalent: 
\begin{enumerate} 
\item $(\pi, E)$ is of moderate growth. 
\item $(\pi, E)$ is an $F$-representation. 
\end{enumerate}
\end{lem}

\begin{proof} By definition any $F$-representation is of moderate growth. 
\par Conversely, assume that $(\pi,E)$ is of moderate growth and let 
$p, q$  and $N>0$ be as in the definition above. 
Then 
$$\tilde p (v):= \sup_{g\in G} {p(\pi(g)v )\over s(g)^N}$$ 
defines a semi-norm on $E$ such that 

\begin{itemize}
\item $p\leq \tilde p\leq q$. 
\item $\tilde p(\pi(g)v)\leq s(g)^N  \tilde p(v)$ for all 
$g\in G$.   
\end{itemize}
The first bulleted item implies that the semi-norms $\tilde p$ define the topology on 
$E$. The  second bulleted item yields that $\tilde p$ 
is $G$-continuous and $[s]$-bounded. 

\end{proof}

\subsection{Smooth vectors and smooth representations}

\subsubsection{Smooth vectors}

\begin{dfn} Let $(\pi,E)$ be an  $F$-representation of $G$. We call a vector 
$v\in E$ {\it smooth} if 
$\gamma_v$ is a smooth map. We denote  by $E^\infty$ the 
vector space of all smooth vectors. 
\end{dfn}

\begin{rem} It is common to define smooth vectors for arbitrary representations
$(\pi, E)$: one says $v\in E$ {\it smooth} provided 
$\gamma_v$ is a smooth map \cite{Bruhat}.  If $(\pi, E)$ is not an 
$F$-representation then this leads to counterintuitive examples: 
\begin{enumerate}
\item The 
regular action of a compact group $G$ on the space of distributions 
$E=C^{-\infty}(G)$ would be smooth. More generally, if $(\pi, \H)$ is a 
Hilbert representation of $G$ and $\H^{-\infty}$ the topological 
dual of $\H^\infty$, then $\H^{-\infty}$ would define a smooth
representation.
\item Let $(\pi, E)$ be a Banach representation and $E^\omega$ the space 
of analytic vectors with its natural inductive limit topology. The dual 
strong dual $E^{-\omega}$ of $E^\omega$, the space of hyperfunction, 
is a Fr\'echet space. The induced action of $G$ on $E^{-\omega}$
would be a smooth representation. 
\end{enumerate}
\end{rem}

Note that  ${\mathcal U}(\gf)$, the universal 
enveloping algebra of the Lie algebra $\gf$ of $G$, acts naturally on $E^\infty$. 
As customary we denote this algebra action by $d\pi$. 
\subsubsection{Sobolev semi-norms}
For a continuous semi-norm $p$ on $E$ we wish to associate a family of Sobolev semi-norms 
$(p_k)_{k\in \N_0}$. We proceed as follows: Fix a basis $X_1, \ldots, X_n$ of $\gf$. 
For all $k\in \N_0$ and $v\in E^\infty$ we set  

$$p_k(v):=\Big[\sum_{m_1+\ldots+m_n\leq k} p(d\pi (X_1^{m_1}\cdot \ldots \cdot X_n^{m_n})v)^2 \Big]^{1\over 2}$$
and refer to $p_k$ as a $k$-th Sobolev norm of $p$. 

\begin{rem} (a) The definition of $p_k$ depends on the choice of the basis $X_1, \ldots, X_n$. However, 
a different basis yields an equivalent semi-norm.
\par\nin (b) If $p$ is $G$-continuous (resp. Hermitian), then so is $p_k$ for any $k\in \N_0$. 
\end{rem}

\subsubsection{Smooth representations}
In the sequel we view $E^\infty$ as a topological vector space with the
locally convex topology induced by all Sobolev semi-norms.  

\begin{dfn} An $F$-representation $(\pi, E)$ is called {\it smooth}
if $E=E^\infty$ holds as topological vector spaces. 
\end{dfn}

Let us denote by $C^\infty(G, E)$ the space of $E$-valued smooth 
functions. We endow $C^\infty(G, E)$ with the topology of smooth compact 
convergence.  We let $G$ act on $C^\infty(G, E)$ as 

$$g\cdot f(x):= \pi(g) f(g^{-1}x) \qquad (g, x\in G; f\in C^\infty(G, E)$$ 
and note that this action is continuous. 
Hence the space of $G$-invariants $C^\infty(G, E)^G$ is a closed subspace 
of $C^\infty(G,E)$.  The following standard fact is found in \cite{Bruhat}. 

\begin{lem} Let $(\pi, E)$ be an  $F$-representation of $G$.
Then the  map 
$$E^\infty \to C^\infty(G, E)^G, \ \ v\mapsto \gamma_v$$
is an isomorphism of topological vector spaces. 
In particular, $E^\infty$ is complete. 
\end{lem}

In the sequel we call a smooth $F$-representation simply 
$SF$-representation. 

\begin{cor} Suppose that $(\pi, E)$ is an F-representation. Then 
$(\pi,E^\infty)$ is an SF-representation of $G$.\end{cor}

\subsection{Integration of representations and algebra actions}

Let us  denote by $\M(G)$ the Banach space of complex Borel measures
on $G$.  We recall that $\M(G)$ carries a natural Banach algebra structure
by convolution of measures:
$$(\mu*\nu) (f):= \mu_x (\nu_y(f(yx)))\,  $$
for $\mu, \nu\in \M(G)$ and $f\in C_c(G)$. 
We denote by $\M_c(G)\subset \M(G)$ the subalgebra 
of compactly supported complex measures. 

\begin{rem} The left action of $G$ on $G$ induces an action of $G$ 
on $\M(G)$ by isometries. This natural action is not continuous, i.e.
does not define a representation.   
Call a measure $\mu$ {\it continuous} provided the orbit map 
$$G\to \M(G), \ g \mapsto (\lambda_g)_* \mu $$
is continuous. Here $\lambda_g(x)= gx $ for $x\in G$ is the left translate. 
Let us denote by $\tilde \M(G)$ the space of continuous complex measures. 
If we fix a left Haar measure $dg $ on $G$, 
then the  map 
$$L^1(G) \to \tilde \M(G), \ \ f\mapsto f\cdot dg $$
provides an isomorphism of Banach algebras. 
\end{rem}

\par If $(\pi,E)$ is  representation of $G$ on a complete topological 
vector space,  then we denote by $\Pi$ the corresponding algebra representation of $\M_c(G)$: 

\begin{equation} \label{g4} \Pi(\mu)v=\int_G \pi(g)v \ d\mu(g) 
\qquad (\mu \in \M_c(G) \, v\in E)\, .\end{equation}
Note that the defining vector valued integral converges 
as $E$ is complete. 

\par Depending on the type of the representation $(\pi, E)$ larger algebras 
as $\M(G)$ might act on $E$. For instance if $(\pi, E)$ is a
bounded Banach representation, then $\Pi$ extends to a representation of 
$\M(G)$. The natural algebra acting  on an $F$-representation
is the algebra of rapidly decreasing complex measures  on $G$. 

\par The space  of {\it rapidly decreasing continuous complex measures} 
on $G$ is defined as 

$$\cR(G):=\{ \mu \in \tilde \M(G)\mid (\forall n\in \N) \  s(g)^n \in L^1(G, |\mu|)\}\, .$$
Let us emphasize that $\cR(G)$ only depends on the scale structure $[s]$. 
We write $L\times R$ for the regular representation 
of $G\times G$ on functions  on $G$: 
$$(L\times R)(g_1, g_2) f (g):= f(g_1^{-1} g g_2)$$
for $g, g_1, g_2\in G$ and $f\in C(G)$.
The following properties of $\cR(G)$ are easy to verify:
\begin{itemize}
\item $(L\times R, \cR(G))$ is an $F$-representation of $G\times G$. 
\item $\cR(G)$ is a Fr\'echet algebra under convolution. 
\item Any F-representation $(\pi, E)$ of $G$ integrates to a continuous 
algebra representation 
\begin{equation}\label{g3} \cR(G) \times E \to E, \ \ (\mu, v)\mapsto 
\Pi(\mu)v, \end{equation}
i.e. the $E$-valued integrals in (\ref{g4}) converge absolutely, the bilinear map (\ref{g3}) is 
continuous and $\Pi(\mu*\nu)=\Pi(\mu)\Pi(\nu)$ holds for all 
$\mu, \nu\in \cR(G)$. 
\end{itemize}

\par For $u\in \U(\gf)$ we will abbreviate 
$L_u:=dL(u)$ and likewise $R_u$ for the derived representations. The smooth vectors 
of $(L\times R, \cR(G))$ constitute the {\it Schwartz space} 

\begin{align*} \S(G):=\{ f\cdot dg \mid f\in C^\infty(G); 
& \forall u, v\in \U(\gf), 
\forall n\in \N\\ 
&s(g)^nL_u R_v f \in L^1(G)\}\, .
\end{align*}
It is clear that $\S(G)$ is a Fr\'echet subalgebra of $\cR(G)$
(see \cite{W}, Sect. 7.1 for a discussion in a wider context if $G$ is reductive). 

\begin{rem}\label{rem=r=s} Suppose that $[s]$ is the maximal or algebraic scale structure on $G$. 
Then for a function $f\in \cR(G)$ the following 
assertions are equivalent: (1) $f$ is in $\S(G)$, i.e. $f$ is 
$L\times R$-smooth; (2) $f$ is $R$-smooth; (3) $f$ is $L$-smooth. 
In fact,  a left derivative $L_u$ at a point $g\in G$ 
is the same as a right derivative $R_{\Ad(g)^{-1}u}$ at $g$. Now observe 
that $\|\Ad(g)\|\leq C s(g)^C $ for all $g\in G$ and a fixed $C>0$. 
\end{rem}

The natural algebra to consider for a SF-representation is 
the Schwartz algebra $\S(G)$ (see Proposition \ref{p=eq} below).  

\begin{rem} If $(\pi, E)$ is a smooth Fr\'echet-representation, then\linebreak  
$\Pi(C_c^\infty(G))E=E$ by Dixmier-Malliavin \cite{DM}. 
Assume in addition that $(\pi, E)$ is a  SF-representation. As $\cR(G)$ acts on $E$ and 
$\cR(G)\supset \S(G)\supset C_c^\infty(G)$, 
we deduce  that $\Pi(\S(G))E=\Pi(\cR(G))E=E$. 
\end{rem}

If $\A$ is an algebra without $\1$ and $M$ is an $\A$-module, then we call $M$ {\it non-degenerate} provided 
$\A M=M$.

\begin{prop}\label{p=eq} Let $G$ be a Lie group. Then the following categories are equivalent: 
\begin{enumerate} 
\item The category of SF-representations of $G$. 
\item The category of non-degenerate continuous algebra representations of $\S(G)$ on 
Fr\'echet spaces.
\end{enumerate}
\end{prop}
 
\begin{proof} We already saw that every SF-representations $(\pi, E)$ gives rise to a non-degenerate continuous 
algebra representation $(\Pi, E)$ of $\S(G)$. 
Conversely let $(\Pi, E)$ by a continuous non-degenerate algebra representation of $\S(G)$. 
Let us denote by $S(G)\hat \otimes_\pi E$ the projective tensor product of $\S(G)$ and $E$. Clearly 
$S(G)\hat \otimes_\pi E$ is a Fr\'echet space and we define an SF-module structure for $G$ by 
$$ g\cdot (f\otimes v):= L(g)f\otimes v \qquad (g\in G, f\in \S(G), v\in E)\, .$$
As $\Pi$ is non-degenerate, $E$ becomes a quotient of $S(G)\hat \otimes_\pi E$  and Lemma \ref{lem=222} completes the proof. 
\end{proof}

\section{Basic representation theory II: Banach representations}

In this section we investigate Banach representations, in particular 
we are interested in the fine Sobolev structure of smooth vectors. 
We view Banach representations as appropriate local models for 
F-representations and point out that the main results in this section 
hold for F-representations as well.

\subsection{Contragredient  representations}\label{DR} 
Throughout this section we let $E$ denote a Banach space. We denote 
by $E^*$ its topological dual.  
We fix a norm $p$ on $E$ and note that $E^*$ is Banach space 
with respect to the dual norm 

$$p^*(\lambda):=\sup_{p(v)\leq 1} |\lambda(v)| \qquad (\lambda\in E^*)\, .$$

\par Let $(\pi, E)$ be a Banach representation of $G$ and 
consider the group homomorphism 

$$\pi^*: G \to \Gl(E^*), \ \ \pi^*(g)(\lambda):=\lambda \circ \pi(g^{-1})\, .$$
From the local equicontinuity of $\pi$ the local equicontinuity of $\pi^*$
follows (see \cite{T}, Ch. 19). However, orbit maps for $\pi^*$ might fail to
be continuous as we will see in an example below.

\par The fact that orbit maps for $\pi^*$ might fail to be continuous
can be dealt with in the following way. Let us consider 
the subspace $\tilde E\subset E^*$  consisting of those 
vectors 
$\lambda\in E^*$ for which 
the orbit map $\gamma_{\lambda}: G\to E^*$ is continuous
(we call this space the continuous dual). Lemma \ref{repcrit}
implies that $\tilde E$ is a closed $G$-invariant subspace of $E^*$.
Following \cite{Bruhat} we restrict the action of $G$ to 
this subspace and obtain  a representation $(\tilde \pi, \tilde E)$ 
that we call the 
{\it contragredient representation}
of $(\pi, E)$.

\par Let us denote by $\H(G)=C_c^\infty(G)\cdot dg$ the algebra 
of smooth compactly supported measures on $G$. 
The standard technique of Dirac approximation
yields:

\begin{lem} Let $(\pi, E)$ be a Banach representation of $G$. Then the following
assertions hold:
\begin{enumerate}
\item For all $\mu\in \H(G)$ the operator $\Pi^*(\mu)$ is defined on $E^*$ 
and maps $E^*$ into $\tilde E$. 
\item The spaces $\Pi^*(\H(G))(\tilde E)$ and $\Pi^*(\H(G))E^*$ are 
dense in $\tilde E$. 
\end{enumerate}
\end{lem}

\par We write $\tilde p$ for the restriction of $p^*$ to $\tilde E$.
Consider the natural isometric morphism $E\to E^{**}$. 
The inclusion $\tilde E  \to E^*$ yields a contractive 
projection $E^{**} \to (\tilde E)^*$ and hence a contractive map  
$i: E\to \tilde {\tilde E}\subset (\tilde E)^*$. 
 
\begin{prop} \label{neq} Let $(\pi, E)$ be a Banach representation 
of $G$ and $p$ be a defining norm  on $E$. Then 
the natural morphism 
$i: E\to \tilde {\tilde E}$ is an isometric embedding.  
\end{prop}

\begin{proof} We need to show that 

\beq \label{inorm} {\tilde p}^* (v) =p (v) \qquad (v\in E)\, .\eeq
As $E^{**} \to (\tilde E)^*$ is contractive, the inequality 
``$\leq$'' follows. 
As for the reverse inequality let us fix a unit vector $v\in E$. 
By Hahn-Banach, we find $\lambda\in E^*$ with $p^*(\lambda)=1$ such that 
$\lambda(v)=1$. Let $\e>0$ and choose a non-negative normalized 
$\mu\in \H(G)$ such that 
$p(\Pi(\mu)v -v)<\e$.  Set $\xi:=\Pi^*(\mu^\vee)(\lambda)$, where 
$\mu^\vee$ is the push-forward of $\mu$ under $g\mapsto g^{-1}$.  
Then  $\xi\in \tilde E$  by the previous lemma. 
If we choose $\supp (\mu)$ small enough such that $\|\xi\|\leq 1+\e$, then 
$$|\xi(v)|= |1 + \lambda (\Pi(\mu)v - v)|\geq 1 -\e$$
and the proof is complete. 
\end{proof}

\par Let us call a Banach representation $(\pi, E)$ 
{\it reflexive} if the morphism 
$i: E\to \tilde{\tilde E}$ is an isomorphism. 

Note that Proposition \ref{neq} shows that 
$(\pi, E)$ is reflexive if $E$ is reflexive 
(see also \cite{War}, Cor. 4.1.2.3). The converse is not true 
as the following example shows.

\begin{ex} Let $G$ be a compact Lie group and 
$(\pi, E)$ be the left regular representation 
of $G$ on $E=L^1(G)$. Then it is easy to check that: 
\begin{enumerate} \item  $\tilde E = C(G)$ while $E^*= L^\infty(G)$. 
\item $\tilde{C(G)}=L^1(G)$ while $C(G)^*$ is the space of Radon measures on 
$G$. 
\end{enumerate}
This shows that a representation $(\pi, E)$ can be reflexive while the Banach 
space $E$ is not reflexive. 
\end{ex}

\subsection{Induced Sobolev norms}

Our definition of smooth vectors for a representation $(\pi, E)$ 
was a geometric one: we said $v\in E$ is smooth if the orbit map 
$\gamma_v: G\to E$ was smooth.  Further we put a Fr\'echet topology 
on $E^\infty$ via the Sobolev norms $p_k$ which turned $(\pi, E^\infty)$
into an $SF$-module for $G$.

\par In this paragraph  we will introduce a new class of Sobolev norms 
on $E^\infty$ which are more quantitative and easier to work with.  

\par To begin with we associate to  $\lambda\in E^*$ and $v\in E$ 
the matrix coefficient
$$m_{\lambda, v}(g):= \lambda(\pi(g)v) \qquad (g\in G)$$
which is a continuous function on $G$ and smooth provided 
$v$ is smooth. 
We fix a relatively compact neighborhood $B$ of $\1$ in $G$
and a test function $\phi$ on $G$ which is supported in $B$ and 
positive near $\1$. 
 
\par We denote by $\F$ the space of test function on $G$ which are 
supported in $B$. For a continuous semi-norm $q$ on $\F$ 
we define a semi-norm $p_q$ on $E^\infty$, the {\it semi-norm induced by $q$}, by 

$$p_q(v):=\sup_{p^*(\lambda)\leq 1}  q(\phi \cdot m_{\lambda, v}) \qquad
(v\in E^\infty)\,.  $$
Note that the choices of both $B$ and $\phi$ for the
  definition of $Sp_s$ are irrelevant; other choices yield equivalent norms. 
\par Typical examples we have in mind 
for $q$ are $L^p$-Sobolev norms or semi-norms of the form 
$q(f)=\sup_{g\in C_B} |Df(g)|$ where $D$ is a differential operator and $C_B$ a 
compact subset of $B$.  
Since $m_{\lambda, v}(g)=m_{\lambda, \pi(g)v}(\1)=m_{\tilde \pi(g^{-1})\lambda, v}(\1)$ we conclude that 
the semi-norms $p_q$ are $G$-continuous. 
In the special case where $q$ is the $L^2$-Sobolev norm on $\F$ for order $s\in \R$
we write $Sp_s$ instead of $p_q$.

\begin{lem}\label{basicsob} Let $(\pi, E)$ be a Banach representation of $G$. 
Then the following assertions hold: 
\begin{enumerate} 
\item For all $k\in \N_0$ there exists a constant $C_k>0$
such that 
$$Sp_k(v)\leq C_k \cdot p_k (v) \qquad (v\in E^\infty)\, .$$
\item  For all $k\in \N_0$ and $s> k+ \dim G/ 2$ 
there exists a constant $c_s>0$
such that 
$$p_k(v) \leq c_s \cdot Sp_s(v) \qquad (v\in E^\infty)\, .$$
\end{enumerate}
\end{lem}

\begin{proof} The first assertion is obvious and the second is 
the standard Sobolev Lemma. 
\end{proof}

\subsubsection{Laplace Sobolev norms}

For our fixed basis $X_1, \ldots, X_n $ of $\gf$ 
we  define a Laplace-element
$$\Delta:=X_1^2 + \ldots + X_n^2 \in \U(\gf)\, .$$ 
If $p$ is a  defining norm of $E$, then we set 

$$^{\Delta} p_{2k}(v):=\Big( \sum_{j=0}^k p( d\pi(\Delta^j) v)^2\Big)^{1\over 2}  \qquad (v\in E^\infty)\, .$$
We will refer to ${}^\Delta p_{2k}$ as the 2k$th$ {\it Laplace-Sobolev norm}  
of $p$. Note that ${}^\Delta p_{2k}$ is equivalent  to a norm  induced from $q$ where 
$$q(f)^2  =\sum_{j=0}^k |\Delta^j f(\1)|^2\, .$$

The next result is an immediate consequence of Lemma \ref{basicsob}
and was motivated by \cite{GKL}, Rem. 5.6 (b). 

\begin{prop}\label{l=gkl} Let $G$ be Lie group and $(\pi, E)$ be a Banach 
representation of $G$.  
Then for all $k\in \N$ there exists 
a $C_k>0$ such that 
$$p_{2k} (v) \leq C_k \cdot {}^\Delta p_{2k + \dim G} (v) \qquad 
(v\in E^\infty)\, .$$
In particular, the topology 
of $E^\infty$ is defined by the family of Sobolev norms 
$({}^\Delta p_{2k})_{k\in \N}$. 
\end{prop}

We can put this into a more general context: For any $s\in \R$ we let 
$E_s$ be the completion of $E^\infty$ with respect to $Sp_s$. 

\begin{lem} For a Banach representation $(\pi, E)$ the following assertions
hold true: 
\begin{enumerate} 
\item $E^\infty=\bigcap_{s>0} E_s$. 
\item For all $s\in \R$ the map 
$$d\pi(\Delta): E_s \to E_{s-2}$$
has closed range. If it is injective, then it is an isomorphism of Banach spaces. 
\end{enumerate}
\end{lem}

\subsection{Action of distributions}

In this paragraph e recall a few facts and notions about compactly supported distributions 
on $G$ and their action on smooth vectors.  

\par Recall that $C^\infty(G)$ carries a natural 
Fr\'echet topology and that its dual, in symbols $\D_c'(G)$, are the 
distributions with compact support.

\par We let $G$ act on  $C^\infty(G)$ by left translation 
in the argument.  
This induces an action of $\U(\gf)$ on $C^\infty(G)$. 
If $u\mapsto u^t$ 
is the canonical anti-automorphism of $\U(\gf)$, we then obtain action 
of $\U(\gf)$ on distributions $\D_c'(G)$ by 

$$(u*T)(f):= T(u^tf) \qquad (u\in \U(\gf), T\in \D'(G), f\in C_c^\infty(G))\, .$$
Note that $\D_c'(G)$ is an algebra under convolution:
for $S, T\in \D_c'(G)$ and $\phi\in C^\infty(G)$ define 
$$(S*T)(\phi):= S_x(T_y(\phi(xy)))\, .$$
By the fundamental theorem of distribution theory every 
$T\in \D_c'(G)$ can be expressed as $T= u* f$ for some $u\in \U(\gf)$ and 
$f\in C_c(G)$. One checks that 
$$\Pi(f)v:= \int_G f(g) \pi(g) d\pi(u)v \ dg \qquad (v\in E^\infty)$$
does not depend on the representation $T= u*f$ and defines
an algebra action of $\D_c'(G)$ on $E^\infty$.

\begin{lem}\label{ereg} {\rm (Elliptic regularity)} Let $B$ be neighborhood 
of $\1$ in $G$. Let $m\in \N$ be such that $2m> \dim G$.
Then 
$$\delta_{\bf 1}= \Delta^m * f_1 + f_2$$
for a $C^{2m -\dim G - 1}$-function $f_1$ and a smooth function $f_2$, both 
supported in $B$. 
\end{lem} 

\begin{proof} 
By local solvability 
and regularity of elliptic PDE (see \cite{H}, Th. 7.3.1 and Cor. 7.3.1) there 
exists a $C^{2m -\dim G -1}$-function $f$ on $G$, smooth on $G\bs\{\1\}$,  
such that $\Delta^m* f=\delta_{\bf 1}$ holds in a small neighborhood 
$V\subset B$  of $\1$ in $G$. 
\par Let $U$ be a relatively compact open neighborhood of $\1$ with
$\oline U\subset V$ and let  $\psi$  be a test function 
with $\psi\mid_U=1$ and $\supp \psi\subset V$. 
Set $f_1:=\psi f$ and $f_2:=\Delta^m(\psi f) - \psi \Delta^m (f) $ and note that $f_1$ and $f_2$ are both supported in $B$ with $f_2$ smooth. 
\end{proof}

\begin{cor}\label{cor=factor} Let $(\pi, E)$ is a Banach representation of $G$. Then
for $u\in \U(\gf)$ and $m\in \N$ as in Lemma \ref{ereg}:
\begin{equation} \label{eeq} v = \Pi(f_1) d\pi(\Delta^m) v + 
\Pi(f_2)v \qquad (v\in E^\infty) \end{equation}
\end{cor}

\subsection{Banach representations for a reductive group}\label{KSOB}
In this paragraph we assume that $G$ is a real reductive group. 
We fix a maximal compact subgroup $K<G$.

\par Recall the notion of Laplace-Sobolev norm ${}^\Delta p_{2k}$  for a
Banach representation $(\pi, E)$. 
For that we fixed a basis $X_1, \ldots, X_n$ to define the Laplacian element $\Delta=\sum_{j=1}^n X_j^2$. 
The choice of the basis is in fact irrelevant and henceforth we will use a specific basis 
which is suitable for us. Such a basis is constructed as follows. Let $\kf$ denote the Lie algebra of $K$ 
and let $\gf=\kf+\pf$ be the associated Cartan decomposition of $\gf$. 
We fix a non-degenerate 
invariant bilinear form $B(X,Y)$ on $\gf$ such that $B$ is negative definite on $\kf$ and positive definite on $\pf$. 
If $\theta:\gf\to\gf$ is the Cartan-involution associated to $\gf=\kf+\pf$, then $\la X,Y\ra=-B(\theta(X), Y)$ is an inner product. 
Our choice of  $X_1, \ldots, X_n$ will be such that $X_1, \ldots X_m$ forms  an orthonormal basis of $\kf$ and 
$X_{m+1}, \ldots, X_n$ is  an orthonormal basis of $\pf$. 

\par With regard to our choice of basis we obtain a central element (Casimir element) in $\cZ(\gf)$ by setting 

$$\cC= - \sum_{j=1}^m X_j^2 +\sum_{j=m+1}^n X_j^2$$
and with $\Delta_\kf:=\sum_{j=1}^m X_j^2$ we arrive at the familiar relation 
\begin{equation} \label{famrel} \cC=\Delta -2\Delta_\kf\, . \end{equation}

\par Let $s\in\N_0$ and define 
the {\it $2s$-th $K$-Sobolev norm} of $p$ by 
$$p_{2s, K}(v)^2:= \sum_{j=0}^s p( \Delta_\kf^s v)^2  \qquad (v\in E^\infty)\, .$$

We claim that $K$-Sobolev semi-norms $p_{s, K}$ can be naturally defined for every 
$s\in \R$.  Heuristically this can be seen as follows: Let $\tf<\kf$ be a
Cartan subalgebra. We fix a notion of positivity on $\tf$ and identify
irreducible representations $\tau$ of $K$  with its highest weight in
$i\tf^*$. For an element $\lambda \in i\tf^*$ we write $|\lambda|$ for its 
Cartan-Killing norm. 
 Now on each $K$-type $E[\tau]$ of $E$ one has 
$$-\Delta_\kf v = (\underbrace{|\tau+\rho_\kf|^2 -|\rho_\kf|^2}_{=:\|\tau\|})v \qquad (v\in E[\tau])$$  
and it is clear how  $\Delta_\kf^s $ should be defined as an operator by breaking a vector $v\in E$ into its 
$K$-Fourier series.  On a more formal level we note that action of $\Delta_\kf^s$ on $E$ is realized by left convolution 
with a distribution $\Theta_s\in C^{-\infty}(K)$ on $K$. Hence the fact that 
$C^{-\infty}(K)$ acts continuously on every SF-module yields our claim.

\begin{prop}\label{lem=K}  Let $(\pi, E)$ be a Banach representation of a real
reductive group $G$. Suppose that one of the following conditions
is satisfied:
\begin{enumerate}
\item The linear map $d\pi(\cC): E^\infty\to E^\infty$ extends to 
a morphism on $E$. 
\item There exists a polynomial $P$ such that 
$P(d\pi(\cC))|_{E^\infty}\equiv 0$. 
\end{enumerate}
Then the topology on the SF-module $E^\infty$ is induced by the 
$K$-Sobolev norms norms $(p_{2n, K})_{n\in \N_0}$. 
\end{prop}

\begin{proof} Assume first that $d\pi(\cC)$ extends to a continuous 
linear operator on $E$. Let $v\in E^\infty$ and note 
that: 

$$p(d\pi(\Delta)v)= p(d\pi(C +2\Delta_{\kf})v)\leq C p(v) + 2 
p(d\pi(\Delta_\kf)v)\, . $$ 
The assertion follows with Proposition \ref{l=gkl}.

\par Assume now that the second condition is staisfied 
and let $d$ be the degree of the polynomial $P$. 
As $\Delta=\cC + 2\Delta_\kf$ we get for $k\in \N$, $k>d$,  that 
$$\Delta^k = \sum_{j=0}^m Q_{j,k}(\cC) \Delta_\kf^j$$
with $Q_{j,k}$ polynomials of degree smaller than $d$. 
For $m$ sufficiently large, Lemma \ref{ereg} yields
$ \delta_{\bf 1}  = \Delta^m * f_1 + f_2$
with $f_2$ smooth and $f_1$ in $C^{2m- \dim G- 1}(G)$ and both compactly
supported. In particular we get for $l\in \N_0$

\begin{align*} \Delta^l &= \Delta^{m+l} *f_1 + \Delta ^l*f_2\\
&= \sum_{j=0}^{m+l} \Delta_\kf^j * Q_{j,m+l} (\cC) * f_1  +\Delta^l*f_2
\end{align*}
With $F_j:= Q_{j, m+l}(\cC)*f_1$ and $F=\Delta^l*f_2$ we thus get 
$$ \Delta^l = \sum_{j=0}^{m+l} \Delta_\kf^j * F_j +F\, .$$
If $m$ is large compared to $l$, the $F_j$ are continuous and $F$ is smooth. 
The assertion follows if we apply this identity to a smooth vector. 
 
\end{proof}

For a representation $(\pi, E)$ of $G$ let us write $E_K^\infty$ for the space of smooth vectors for the $K$-representation $\pi\mid_K$. 

\begin{cor}\label{lem=K2} Let $(\pi, E)$ be a Banach representation of a real
reductive group $G$. Suppose that one of the conditions in Proposition
\ref{lem=K}
holds. Then $E^\infty = E_K^\infty$.
\end{cor}

\begin{proof} Set $V:=E^\infty$. 
We claim that $V$ is dense in the 
Fr\'echet space $E_K^\infty$. 
\par In order to prove the claim we first note that 
the Garding subspace
$\Pi_K(C^\infty(K)\cdot dk ) E$ is dense in $E_K^\infty$. Let $\mu \in C^\infty(K)
\cdot dk$ be a smooth measure and $v=\Pi_K(\mu)u$ for some $u\in E$. 
Let $(\nu_k)_{k\in \N} \subset C_c^\infty(G)\cdot dg$ be a Dirac sequence so
that  $u_k:=\Pi(\nu_k)u \to u$ in $E$. Then $v_k:=\Pi_K(\mu)u_k$ is in
$E^\infty$ and converges to $v$ in $E_K^\infty$, establishing our claim.

\par The claim implies that $E_K^\infty$ is the completion of $V$ 
with respect to $(p_{2n, K})_{n\in \N_0}$ and the conclusion follows with 
Proposition \ref{lem=K}.
\end{proof}

\section{Harish-Chandra modules I: algebraic facts}

In this section we will review some central algebraic facts about Harish-Chandra modules.
This is followed by a discussion of basic topological properties of their globalizations
in the following section.  
\par From now on we assume that $G$ is a linear reductive 
group. Let us fix  a maximal compact subgroup $K$ of $G$. 
   
\medskip  By a $K$-module $V$ we shall understand a complex vector space endowed with 
linear algebraic action of $K$, that is $V$ is a union of 
finite dimensional algebraic $K$-representations.

\par If  $(\pi, E)$ is a representation of $K$, then we denote by $E^{K-{\rm fin}}$
the $K$-module consisting of $K$-finite vectors. 
\par We call a $K$-module $E$ {\it weakly admissible}
provided for all finite dimensional representations $(\tau, W)$ of $K$ the 
multiplicity space $\Hom_K(W, E)$ is finite dimensional. 
\par We call a representation $(\pi, E)$ of $G$ {\it weakly admissible} provided  
$E^{K-{\rm fin}}$ is a weakly admissible $K$-module.

\par By a $(\gf, K)$-module $V$ we understand a module for $\gf$ and 
$K$ such that: 
\begin{itemize} 
\item The derived action of $K$ coincides with the 
action of $\gf$ restricted to $\kf:={\rm Lie} K$. 
\item The actions are compatible, i.e.
$$k \cdot X\cdot  v = \Ad(k)X \cdot k \cdot v$$
for all $k\in K$, $X\in \gf$ and $v\in V$. 
\end{itemize}

\begin{rem} (a) If  $(\pi, E)$ is a weakly admissible Banach representation of $G$, then 
$E^{K-{\rm fin}}$ consists of 
smooth vectors and is stable under $\gf$ --  in other words $E^{K-{\rm fin}}$ is a
weakly admissible $(\gf, K)$-module.
\par\noindent (b) Let us emphasize that a weakly admissible $(\gf,K)$-module
is not necessary finitely generated as a $\gf$-module. For example 
the tensor product of two representations of the holomorphic discrete series 
for $(\gf, K)$=$(\sl(2,\R), \SO(2, \R))$ is admissible but not finitely generated as a 
$\gf$-module. 
\end{rem} 

Let us denote by $\cZ(\gf)$ the center of $\U(\gf)$ and by 
$\spec(\cZ(\gf))$ its spectrum, i.e. the set of all algebra characters
$\cZ(\gf)\to \C$.  A $\U(\gf)$-module $V$ will be called $\cZ(\gf)$-finite provided 
there exists an ideal $I\triangleleft \cZ(\gf)$ of finite codimension 
which annihilates $V$, or, equivalently, provided there exists  
$\chi_1, \ldots, \chi_n\in \spec (\cZ(\gf))$ and $d\in \N$ such that 
$\prod_{j=1}^n (\chi_j(z) -z)^d v =0$ for all $v\in V$ and $z\in \cZ(\gf)$.

\par We denote by  $\Irr (\gf, K)$ be the set of equivalence classes of 
irreducible $(\gf, K)$-modules. As customary we will not distinguish between 
equivalence classes and their representatives. As every $V\in \Irr(\gf, K)$
is of countable dimension, Dixmier's version of Schur's Lemma is applicable 
(see \cite{W}, 0.5.1 and 0.5.2) and associates to $V$ an {\it infinitesimal 
character} $\chi_V \in \spec(\cZ(\gf))$. 
\par Harish-Chandra study of distributional characters of irreducible admissible representations
led him to the following fundamental result:

\begin{thm}\label{HC11}  {\rm(Harish-Chandra)} 
\begin{enumerate}
\item Every $V\in \Irr(\gf, K)$ is weakly admissible. 
\item The map 
$$\Irr(\gf, K)\to \spec(\cZ(\gf)), \ \ V\mapsto \chi_V$$
has finite fibers. 
\end{enumerate}
\end{thm}
Harish-Chandra's theorem allows us to characterize finitely 
generated weakly admissible modules in various useful ways: 

\begin{thm}\label{th=HC} For a weakly admissible  $(\gf,K)$-module $V$ the following assertions
are equivalent:
\begin{enumerate}
\item $V$ is finitely generated as a $\gf$-module. 
\item $V$ is $\cZ(\gf)$-finite. 
\item $V$ is finitely generated as an $\nf$-module,  where 
$\nf$ is a maximal unipotent subalgebra of $\gf$. 
\end{enumerate}
\end{thm}
\begin{proof}(i)$\Rightarrow$(ii) follows from the fact that we can take 
generators of $V$ belonging to $K$-types and the fact that $\cZ(\gf)$ preserves
$K$-types. Harish-Chandra's Theorem \ref{HC11} implies  (ii)$\Rightarrow$(i). 
\par  The implication (iii)$\Rightarrow$(i) is clear. Finally, a result of Osborne (\cite{W}, Prop. 3.7) asserts that 
$$\U(\gf)= \U(\nf) F \cZ(\gf) \U(\kf)$$
for a finite dimensional subspace $F\subset \U(\gf)$. 
Thus (i) and (ii) together imply (iii). 
\end{proof}

A $(\gf,K)$-module $V$ will be called 
{\it Harish-Chandra module} or {\it admissible $(\gf, K)$-module}
if the conditions in the theorem above are satisfied. 
Likewise, we will call a smooth Fr\'echet representation $(\pi, E)$ {\it admissible} provided the 
underlying $(\gf, K)$-module $E^{K-{\rm fin}}$ is admissible.  

\par Harish-Chandra modules form a category $\HC$ with morphisms the linear $(\gf, K)$-maps. 
This category has a natural duality structure which we are going to describe now. 
If $V$ is a Harish-Chandra module, then we denote by $V^*$ its 
algebraic dual and by $\tilde V\subset V^*$ the $K$-finite vectors 
in $V^*$. Note that $\tilde V$ is a $\gf$-submodule of $V^*$. 
As $V$ is weakly admissible the same holds for $\tilde V$ and we  readily obtain that 
$$\tilde {\tilde V}=V\, .$$
As $V$ is $\cZ(\gf)$-finite, the same holds for $\tilde V$ and thus 
$\tilde V$ is again a Harish-Chandra 
module by Theorem \ref{th=HC} above. We refer to $\tilde V$ as the {\it Harish-Chandra 
module dual to $V$}. 

\par Another basic feature of Harish-Chandra modules is the Casselman embedding theorem 
which asserts that every Harish-Chandra module can be embedded into the $K$-finite vectors 
of some minimal principal series representation (\cite{Cas}). To be more specific, let 
us fix an Iwasawa decomposition $G=NAK$ and write 
$P_{\rm min}=MAN$ for the associated minimal parabolic 
subgroup. Here $M:=Z_K(A)$ is the centralizer of $A$ in $K$.
For a finite dimensional $P_{\rm min}$-module $W$ we denote by 
$$I^\infty(W):=C^\infty(W\times_{P_{\rm min}} G)$$
the smooth sections of the $G$-equivariant 
vector-bundle $W\times_{P_{\rm min}} G \to P_{\rm min}\bs G$, that is 
the smooth functions $f: G\to W$ which satisfy $f(pg)=p\cdot f(g)$ for all $p\in P_{\rm min}$ 
and $g\in G$. We topologize $I^\infty(W)$ with the natural Fr\'echet topology 
of compact convergence of all derivatives. 
Note that $I^\infty(W)$ becomes an admissible SF-module for $G$ under right-displacements in 
the arguments, i.e.
the prescription 
$$R(g)f(x):=f(xg) \qquad (g, x\in G, f\in I^\infty(W))$$
gives rise to an  admissible SF-representation $(R, I^\infty(W))$. 
The corresponding $(\gf, K)$-module 
$$I(W):=I^\infty(W)^{K-{\rm fin}}$$
is a Harish-Chandra module. 

\begin{thm}\label{th=Cas} {\rm (Casselman)} For every Harish-Chandra module $V$ there exists a
finite dimensional $P_{\rm min}$-module $W$ and a $(\gf, K)$-embedding 
$V\to I(W)$. 
\end{thm}

\begin{proof} \cite{W}, Cor. 4.2.4. \end{proof}

Let us remark that the restriction 
morphism 
$$\Res_K:  I^\infty(W) \to C^\infty( W\times_M K), \ \ f\mapsto f|_K$$
is a $K$-equivariant isomorphism of Fr\'echet spaces. The image  $C^\infty( W\times_M K)$ 
is often more convenient to work with as $K$ is a compact manifold. 
It is important to note that the corresponding $G$-representation $(\pi, C^\infty(W\times_M K))$
defined by $\pi(g):= \Res_K \circ R(g) \circ (\Res_K)^{-1}$ completes 
to a Hilbert representation $(\pi, L^2(W\times_M K))$.

\par As a $K$-module $I(W)$ is isomorphic to $\C[W\times_M K]$ and 
this allows us to give a polynomial bound on the $K$-multiplicities
of a Harish-Chandra module. More specifically let us denote by $\hat K$  the set of equivalence classes of 
irreducible unitary representations of $K$. We will identify
an equivalence class $[\tau]\in \hat K$ with a representative $\tau$. 
If $V$ is a $K$-module, then we denote  by $V[\tau]$ its $\tau$-isotypical 
part. Similarly we denote for $v\in V$ by $v_\tau$ 
its $\tau$-isotypical component.

\par Let $\tf$ be the Lie algebra of a maximal torus of $K$. We often identify 
$\tau$ with its highest weight in $i\tf^*$ (with respect to a fixed 
positive system). In particular, $|\tau|\geq 0$ will refer to the Cartan-Killing norm 
of the highest weight $\tau$. 
As a consequence of Theorem \ref{th=Cas} we obtain 
Harish-Chandra's multiplicity bound.

\begin{thm} {\rm (Harish-Chandra)} \label{th=mb} Let $V$ be a Harish-Chandra module. 
Then there exists a $C>0$  such that 
$$\dim V[\tau]\leq C (1+ |\tau|)^{\dim K - \dim M} \qquad (\tau\in \hat K)\, .$$ 
\end{thm}

\section {Harish-Chandra modules II: topological properties}
This section is devoted to topological properties of representations
$(\pi, E)$ whose $K$-finite vectors form a Harish-Chandra module.

\subsection{Definition and existence of globalizations}

Given a Harish-Chandra module $V$ we say 
that a representation $(\pi, E)$ of $G$ is a {\it globalization}
of $V$ provided the $K$-finite vectors $E^{K-{\rm fin}}$ of $E$ are smooth and 
isomorphic to $V$ as a $(\gf, K)$-module.

\par Every  Harish-Chandra module $V$ admits 
a Hilbert globalization $\H$ (and hence an SF-globalization
by taking the smooth vectors in $\H$). In fact, by Theorem \ref{th=Cas},  
we can embed $V$ into some minimal principal series $I(W)$ and the closure of 
$V$ in $L^2(W\times_M K)$ defines a Hilbert globalization 
$\H$ of $V$. Note that $\H^\infty$ coincide with the closure of $V$ in $I^\infty(W)$.

\begin{rem} We caution the reader that there exist 
irreducible Banach representation 
$(\pi, E)$ of $G$ which are not admissible \cite{S}, i.e. they are not
globalizations
of Harish-Chandra modules. However, if $(\pi, \H)$ happens to be 
a unitary irreducible representation, then Harish-Chandra has shown  that 
$\pi$ is admissible.  
\end{rem}

\begin{rem} Let $V$ be a Harish-Chandra module and $(\pi, E)$ 
a Banach globalization. Then 
$$\Pi(\cR(G))V =\Pi(\S(G))V\, .$$
Indeed, by a basic result of Harish-Chandra 
there exists for each $v\in V$ a $K\times K$-finite $h\in 
C_c^\infty(G)$ such that $\Pi(h) v=v$. As $\cR(G)*C_c^\infty(G)\subset 
\S(G)$ the asserted equality is established. 
\end{rem} 

\subsection{The contragredient of a Banach-globalization}

In the introductory section we saw that the dual of a Banach representation
$(\pi, E)$ might not be continuous which brought us to the notion 
of contragredient representation $(\tilde \pi, \tilde E)$. 
Recall that the continuous dual $\tilde E$ was the largest closed 
subspace of $E^*$ on which the dual action is continuous. If $E$ happens to be reflexive
then $\tilde E =E^*$

\par The case where $(\pi, E)$ is a Banach globalization 
of a Harish-Chandra module $V$ is of particular interest to us. 
Here the situation is as follows: 

\begin{lem}\label{cdual} Let $V$ be a Harish-Chandra module and $\tilde V$ its dual. If
$(\pi,E)$ is a Banach globalization of $V$, then $\tilde V\subset \tilde E$. In
particular, $(\tilde \pi, \tilde E)$ is a Banach globalization of $\tilde V$. 
\end{lem}

\begin{proof} Let $(\pi, E)$ be a Banach globalization of $V$. 
For a $K$-type $\tau\in \hat K$ let us consider the projection

$$\pr_\tau: E \to E[\tau]=V[\tau], 
\ \ v\mapsto \dim (\tau) \int_K \oline{\chi_\tau(k)} \pi(k)v \ dk$$
on the $\tau$-isotypical component. Here $\chi_\tau$ refers to the character
of $\tau$. As $\pr_\tau$ is continuous, the first 
assertion $\tilde V\subset \tilde E$ follows. Further,   a $K$-type 
of $\tilde E$ does not vanish on $V$ as $V$ is dense in $E$. 
Thus the $K$-finite vectors of $\tilde E$ are contained in $\tilde V$. 
\end{proof}

\subsection{Nuclear structures on smooth vectors}
The goal of this section is to prove that the smooth vectors of a Banach-globalization 
carry the structure of a nuclear Fr\'echet space. 

\par It is convenient to introduce some useful notation in this regard. 
Let $V$ be a Harish-Chandra module and $p$ be a norm on $V$. Then we say 
that $p$ is a $G$-continuous norm on $V$ provided the completion of 
$V$ with respect to $p$ gives rise to a Banach representation of $G$. 

\par We introduce a preorder on the set of $G$-continuous 
norms on a Harish-Chandra module $V$: in symbols

$$p\prec q :\iff (\exists k\in \N_0, C>0)\   p(v) \leq C q_k(v) \qquad (v\in V)\, .$$
We say that $p$ and $q$ are {\it Sobolev-equivalent}, in symbols $p\asymp q$,  provided $p\prec q$ and $q\prec p$.

To begin with we recall a familiar result from convex analysis
(John's Theorem, see \cite{MS}, Th. 3.3).

\begin{lem}\label{l=finite} Let $(V, \|\cdot\|)$ be a finite dimensional normed vector space 
of dimension $n$. Then there exists $\lambda_1, \ldots, \lambda_n \in V^*$ 
with $\|\lambda_i\|=1$, $1\leq i\leq n$, such that the associated 
Hermitian  form 
$$Q(x):=\sum_{j=1}^n |\lambda_j(x)|^2$$
satisfies 
$$\|\cdot\|^2 \leq Q \leq 2n \|\cdot\|^2\, .$$ 
\end{lem}

\begin{thm} \label{th=nuc} Let $V$ be a Harish-Chandra module and $p$ be a $G$-continuous norm. 
Then the following assertions hold: 
\begin{enumerate}
\item There exists a $G$-continuous Hilbert-norm $q$ such that $p$ is Sobolev equivalent to  $q$. 
\item There exits a $k\in \N$ such that inclusions $(V, p_k) \to (V, q)$ and 
$(V,q)\to (V, p_k)$ are nuclear. 
\end{enumerate}
\end{thm} 

\begin{proof} Let $p$ be a $G$-continuous norm of $E$. 
We recall from Theorem \ref{th=mb} that there is an integer $N>0$ such that 

\begin{equation} \label{mbo} m(\tau):=\dim V[\tau]\leq (1+|\tau|)^N  \qquad (\tau\in \hat K)\, . \end{equation}
According to Lemma \ref{l=finite} 
we find for each $\tau\in \hat K$  a basis $\lambda_1^\tau, \ldots, \lambda_{m(\tau)}^\tau$ 
of $V[\tau]^*$ such that 
the Hermitian form 
$$Q_\tau(v):=\sum_{j=1}^{m(\tau)} |\lambda_j (v)| \qquad (v\in V[\tau])$$
satisfies 
\begin{equation}\label{squeeze}  p^2(v)\leq Q_\tau(v) \leq 2 m(\tau) p^2(v) \qquad (v\in V[\tau])\, .\end{equation}
Define a Hermitian form $Q$ on $V$ by 
$$Q(v):=\sum_{\tau\in \hat K} Q_\tau(v_\tau)\, $$ 
and let $\tilde q(v)$ be the associated norm. 
Let now $k\in 2\N$ and $\tilde q_{k,K}$ be the $k$-th 
$K$-Sobolev norm of $q$. 
For $\tau\in \hat K$ we set 
$$\|\tau\|_k:=\Big(\sum_{j=0}^{k/2} \|\tau\|^{2j}\Big)^{1\over 2}$$
and note that $\|\tau\|_k\asymp \|\tau\|^k$. 
{}For all $v\in V[\tau]$ we have 

$$ \tilde q_{k,K}(v) = \|\tau\|_k \tilde q(v)\, .$$
Since $p(v)\leq \sum_{\tau\in \hat K} p(v_\tau)$,  we obtain for sufficiently 
large $k$  that 

\begin{align*} p(v) & \leq \sum_{\tau\in \hat K} \tilde q(v_\tau) = \sum_{\tau\in \hat K} 
{1\over \|\tau\|_k} \|\tau\|_k\ \tilde q (v_\tau) \\ 
& \leq \Big[\sum_{\tau\in \hat K}  {1\over \|\tau\|_k^{2}} \Big]^{1\over 2} 
\cdot\Big[\sum_{\tau\in \hat K} \|\tau\|_k^{2} Q_\tau (v_\tau) \Big]^{1\over 2} \\
& \leq C \tilde q_{k,K}(v)\end{align*}

The second inequality in (\ref{squeeze}) combined with the multiplicity bound (\ref{mbo})
yields a constant $C>0$ such that 
$$\tilde q(v)\leq  C p_{k,K}(v)  \qquad (v\in V)\, $$
provided $k\in \N$ is taken large enough. 

\par As $\tilde q$ might not be $G$-continuous, we have to address this issue. 
First note that $p_{K,k}\leq C p_k$ and that $p_k$ is $G$-continuous. Thus 
for sufficiently large $c>0$ the prescription
$$q(v):=\left( \int_G \tilde q(\pi(g)v)^2 e^{-c d(g)} \ d g \right)^{1\over 2}$$
defines a $G$-continuous Hilbert norm with $q\leq C p_k$. Local Sobolev 
on the other hand readily yields that $\tilde q\leq C q_k$ for $k\in \N$ sufficiently big. 
This completes the proof of the theorem.
\end{proof}

\begin{cor}\label{Cor=nuc} Let $(\pi, E)$ be an SF-globalization of a Harish-Chandra module
  $V$. 
Then: 
\begin{enumerate} \item 
$E$ is a nuclear Fr\'echet space.
\item The topology on $E$ is determined by a countable family of $G$-continuous $K$-invariant Hilbert 
semi-norms. 
\end{enumerate}
\end{cor}

In view of Theorem \ref{th=nuc} it is no loss of generality to assume that 
a $G$-continuous norm on a Harish-Chandra module to be Hermitian. In addition we will request 
that all norms are $K$-invariant.

\subsubsection{Weighted function spaces}\label{ss1}
This subsection is about natural realizations 
of Harish-Chandra modules in weighted function spaces on $G$.  

\bigskip For $m\geq 0 $ we define the weighted 
Banach-space
$$C(G)_m:=\Big \{ f\in C(G)\mid p^m(f):=\sup_{g\in G} {|f(g)|\over  \|g\|^m}<\infty \Big \}\, .$$
We view $C(G)_m$ as a module for $G$ under the right regular action $R$.
Note that this action might not be continuous in general (take $m=0$ and $G$
not compact). 
From the properties of the norm one readily shows that 

$$p^m(R(g)v)\leq \|g\|^{m} \cdot p^m(v)$$
for all $g\in G$. Thus the action is locally equicontinuous. 
It follows that the smooth vectors 
for this action $C(G)_m^\infty$ define an $SF$-module  for $G$.
Note that  $C(G)_m^\infty\subset C^\infty(G)$ as a consequence of the local 
Sobolev-Lemma.

\par Likewise we associate to $m\geq 0$ the weighted Hilbert-space
$$L^2(G)_m:=L^2(G, \|g\|^{-m} dg)\, .$$ 
Note that the right regular action of $G$ on $L^2(G)_m$ 
defines a Hilbert representation of $G$. Let us 
denote by $h^m$ the corresponding Hilbert norm.

\par Let $k_0>0$ be such that $\int_G \|g\|^{-k_0} \ dg<\infty$. 
Then for all $m\in\R $ one obtains a continuous embedding

\begin{equation}\label{incl} C(G)_{(m-k_0)/2} \to L^2(G)_m \, \end{equation}
or to phrase it equivalently that there exists a constant $C>0$ such that 
 
\begin{equation}\label{hp1} h^m \leq C \cdot p^{(m-k_0)/2}\, .\end{equation}

To obtain inequalities of the reverse kind we shall employ the Sobolev 
Lemma on $G$. It is not hard to show that the derivatives of the norm 
function $\|\cdot \|$ are bounded by a multiple of $\|\cdot \|$.
Hence we obtain constants $C>0$ and $l_0\in \N$ with $l_0$ independent from
$m$ such that 
\begin{equation}\label{hp2} p^m \leq C\cdot h^{2m}_{l_0}\end{equation}
holds on $L^2(G)_m^\infty$.

\par If $V$ is a Harish-Chandra module, then we denote by 
$\Xi\subset V$ a $\cZ(\gf)$-invariant set of generators of minimal 
dimension, say $k$.

\par Let $V$ be a Harish-Chandra module $V$ and  $\Xi \subset \tilde V$ a 
$\cZ(\gf)$-stable set of generators as above.  We fix an inner product on $\Xi$ and let $\xi_1, \ldots\xi_k$ 
be an orthonormal basis of $\Xi$. 
The inner product on $\Xi$ yields an inner product on the dual space $\Xi^*$.
Attached to $\Xi$ we consider the $G$-equivariant embedding 

$$\phi_\Xi: V^\infty\to C^\infty(G) \otimes \Xi^* =C^\infty(G, \Xi^*); \phi_\Xi(v)(g)(\xi):=  m_{\xi, v}(g)$$ 
with $\xi\in \Xi$ and $m_{\xi, v}(g)= \xi(\pi(g)v)$ the corresponding matrix coefficient. 

\par We claim that $\im \phi_\Xi$ lies in some 
 $C_m(G)^\infty \otimes \Xi^*$ for $m$ suitably large. In fact choose a
Banach globalization $(\pi, E)$ of $V$ with norm $q$. Then 

$$\max_{1\leq j\leq k}|m_{\xi_j,v}(g)|\leq C \|g\|^N q(v)\qquad (v\in V^\infty) $$
for suitable constants $N$ and $C>0$. 
Hence $\im \phi_\Xi \subset C_N(G)^\infty \otimes \Xi^*$. 

Let now $E_N$  be the closure of $\phi_\Xi(V^\infty)$ in  $C_N(G) \otimes \Xi^*$. 

\begin{lem} \label{l=app} With the notation from above, $E_N$ defines 
a Banach globalization of $V$. 
\end{lem}
\begin{proof} It is clear that $E_N$ is a Banach space.
With regard to the norm on $E_N$ the operators 
$\pi(g)$ are bounded by $\|g\|^N$. Hence the action is locally
equicontinuous. 
 Further, as $p_N$ is dominated 
by $q$ on $V^\infty$ we conclude that all orbit maps
$\gamma_v:G\to E_N$ are continuous. Thus $G\times E_N\to E_N$ is 
a representation by Lemma \ref{repcrit}.
\end{proof}

From our construction it is clear that the smooth vectors 
$E_N^\infty$ for $E_N$ coincide with the SF-closure 
$\phi_\Xi (V^\infty)$ in  $C_N(G)^\infty \otimes \Xi^*$. 
Let us denote the restriction of $p^N$ to $E_N$
by the same symbol.

\par Set $N':=2N+k_0$. Then there is a natural 
$G$-equivariant embedding 

$$\psi_\Xi: V^\infty\to L^2(G)_{N'}\otimes \Xi^*;\ \psi_\Xi (v)(g)(\xi)= m_{\xi,v}(g)$$ 
and the closure of the image defines a Hilbert globalization 
$\H_{N'}$ of $V$.

\subsubsection{The dual of an SF-globalization}
The material in this subsection is not needed in the sequel of this article. However 
it contains a fact worthwhile which is worthwhile to mention and thematically fits 
in our discussion. 
\par Let $(\pi, E)$ be an SF-globalization of a Harish-Chandra module $V$. 
In Corollary \ref{Cor=nuc} we have shown that 
$E$ is a nuclear Fr\'echet space. 
As nuclear Fr\'echet spaces are reflexive, it follows that 
the dual representation $(\pi^*, E^*)$ of $(\pi, E)$ exists
and that the bi-dual representation $(\pi^{**}, E^{**})$ 
is naturally isomorphic to $(\pi, E)$.
 
\par For any SF-representation $(\pi, E)$
we recall that the natural action of $\S(G)$, 

$$\S(G)\times E\to E, \ \ (f, v)\mapsto \Pi(f)v$$
is continuous, i.e. a continuous bilinear map (Proposition \ref{p=eq}). 

\par On the dual side we obtain a dual action 
$$\S(G)\times E^*\to E^*, \ \ (f, \lambda)\mapsto \Pi^*(f)\lambda; 
\ \Pi^*(f)\lambda=\lambda\circ \Pi(f^*)$$
where $f^*(g)=f(g^{-1})$. 
For a general SF-representation on a reflexive Fr\'echet space 
the dual dual action $\Pi^*$ of $\S(G)$ might not be continuous.
For globalizations however matters behave well and we record:   

\begin{lem} Let $(\pi, E)$ be an SF-globalization of a Harish-Chandra 
module $V$. Then the dual algebra action $\S(G)\times E^*\to E^*$ 
is continuous. 
\end{lem}

\begin{proof}  We recall from Corollary \ref{Cor=nuc} that 
$(\pi, E)=\displaystyle \quad \varprojlim (\pi_n, E_n)$ is a projective limit of Hilbert 
representations $(\pi_n, E_n)$. Thus $(\pi^*, E^*)={\mathrm{\varinjlim}} (\pi_n^*, E_n^*)$
is a direct limit of Hilbert representation. As for all $n\in\N$ the dual action 
$\S(G) \times E_n^* \to E_n^*$ is continuous, the assertion follows.
\end{proof}

\section{Minimal and maximal SF-globalizations of Harish-Chandra modules}

\par Let us introduce a preliminary notion and call a Harish-Chandra module $V$ {\it good}
if it admits a unique SF-globalization. 
Eventually it will turn out that all Harish-Chandra modules are good 
(Casselman-Wallach). 

\par As we will see below there are  two natural 
extremal  SF-globalizations
of a Harish-Chandra module,  namely {\it minimal} and {\it maximal} 
SF-glo\-ba\-lizations. Eventually they will coincide but they are useful 
objects towards a proof of the Casselman-Wallach theorem. 

\subsection{Minimal globalizations}

\par An SF-globalization, say $V^\infty$, of an Harish-Chandra module $V$
will be called {\it minimal} if the following universal property holds: if $(\pi,E)$ is an SF-globalization of $V$, then there 
exists a continuous $G$-equivariant map $V^\infty\to E$ which extends 
the identity morphism $V\to V$. 
\par It is clear that minimal globalizations are unique. 
Let us show that they actually exist. We need to collect some facts 
about matrix coefficients. In this context we record the following result (see \cite{Cas}):  

\begin{lem} \label{Cas} Let $(\pi, E)$ be a Banach globalization of a Harish-Chandra 
module $V$. Then for all $\xi\in \tilde V\subset E^*$ and $v\in V$, 
the matrix coefficient 
$$m_{\xi, v}(g)=\xi(\pi(g)v) \qquad (g\in G)$$
is an analytic function on $G$. In particular $m_{\xi, v}$ is independent 
of the particular Banach globalization $(\pi, E)$ of $V$. 
\end{lem}

We will now give the construction of the minimal globalization 
of a Harish-Chandra module $V$. For that let us fix a Banach globalization $(\pi,E)$ of $V$. 
Let ${\bf v}  =\{v_1, \ldots, v_k\}$ be a set of generators of $V$ 
and consider the map 

$$\S(G)^k\to E, \ \ {\bf f}=(f_1, \ldots, f_k)\mapsto \sum_{j=1}^k \Pi(f_j)v_j\, .$$
This map is linear,  continuous and $G$-equivariant 
(with $\S(G)^k$ considered as a module for $G$ under the left regular representation). 
Let us write 
$$\S(G)_{\bf v }:=\{{\bf f}\in \S(G)^k \mid \sum_{j=1}^k \Pi(f_j)v_j=0\}$$
for the kernel of this linear map. Note that   $\S(G)_{\bf v}$ is 
a closed $G$-submodule of $\S(G)^k$. 
We claim that $\S(G)_{\bf v}$ is independent of the choice 
of the particular globalization $(\pi, E)$ of $V$: In fact, for 
$v\in V$ and $f\in \S(G)$ we have $\Pi(f)v=0$ if and only if 
$\xi(\Pi(f)v)=0$ for all $\xi\in \tilde V$. 
As $g\mapsto m_{\xi, v}(g)=\xi(\pi(g)v)$ is analytic 
and hence independent of $\pi$ (Lemma \ref{Cas}), the claim follows. 
\par Lemma \ref{lem=222} shows that $\S(G)^k/\S(G)_{\bf v}$ is an SF-module for $G$.
Since $\Pi(\S(G)^{K\times K})V=V$ for $\S(G)^{K\times K}$ the $K\times
K$-finite functions of $\S(G)$, it follows that $\S(G)^k/\S(G)_{\bf v}$ is 
an SF-globalization of $V$.  
By construction 
$\S(G)/\S(G)_{\bf v}$ is the minimal globalization $V^\infty$.

We record the following general Lemma on quotients of Harish-Chandra 
modules in relation to minimal globalizations.

\begin{lem} \label{ll1} Let $V$ be a Harish-Chandra module and $V^\infty$ its 
unique minimal SF-globalization. Let $W\subset V$ be a submodule and 
$U:=V/W$. Let $\oline W$ be the closure of $W$ in $V^\infty$. 
Then $U^\infty= V^\infty/\oline W$. 
\end{lem}

\begin{proof} Let us write $(\pi_U, V^\infty/ \oline W)$ for the 
quotient representation obtained from $(\pi, V^\infty)$. Then 
$\Pi(\S(G))V=V^\infty$ implies that $\Pi_U(\S(G))U=V^\infty/\oline W$ and 
hence the assertion. 
\end{proof}

\begin{rem} Suppose that a Harish-Chandra module $V$ admits a maximal 
$G$-continuous norm $p$ with respect to our Sobolev ordering 
$\prec$. Then $V^\infty$ coincides with the smooth vectors of the 
Banach completion of $(V,p)$. 
However, we want to emphasize that the existence of a minimal 
globalization does not automatically imply the existence of a maximal 
norm. 
\end{rem}

\subsection{Dual norms} \label{DNorms}

Let $q$ be a $G$-continuous Hilbert norm on a Harish-Chandra module $V$. The associated 
Sobolev norms $(q_n)_{n\in \N}$ induce an SF-structure on $V$.

\par Recall the notion $q^*$ of dual norm. Our discussion from 
Subsection \ref{KSOB}, then readily yields:

\begin{lem}\label{l=asymp}  Let $q$ and $p$ be $G$-continuous Hermitian  norms on a Harish-Chandra module $V$. 
Then $q\asymp p $ if and only if $q^*\asymp p^*$. 
\end{lem}

\subsection{Maximal Globalizations}

Let us call  an SF-globalization of $V$, say $V_{\rm max}^\infty$,  {\it maximal}
if for any SF-globalization $(\pi, E)$ of $V$ there exists 
a continuous linear $G$-map $E\to V_{\rm max}^\infty$ sitting above 
the identity morphism $V\to V$. 

\par It is clear that maximal globalizations are unique 
provided that they exist. Moreover, in case a maximal globalizations 
of a Harish-Chandra module $V$ exists, then 
$V$ is good if and only if $V^\infty=V_{\rm max}^\infty$. 

\par Let us emphasize that a maximal globalization of $V$ exists 
if and only if there exists a $G$-continuous Hilbert norm $q$ such that 
$q\prec p$ for all $G$-continuous norms $p$ on $V$. 
Since a module $V$ is good if and only if $p\asymp q$ for all 
$G$-continuous norms $p$ and $q$ on $V$, we obtain 
from Lemma \ref{l=asymp} that: 

\begin{lem} A Harish-Chandra module $V$ is good if and only if its
dual $\tilde V$ is good. 
\end{lem}

With the elementary tools developed so far we cannot give a construction of 
maximal globalizations, but we would like to emphasize that the existence of 
maximal globalization would be implied by the matrix coefficient
bounds established in \cite{W}, Sect 4.
For us the following criterion will be sufficient:

\begin{lem}\label{ll2} Let $U$ be a good Harish-Chandra module and 
$U^\infty$ its 
unique SF-globalization. Let $V\subset U$ be a submodule and 
let $\oline V$ be the closure of $V$ in $U^\infty$. Then 
$V_{\rm max}^\infty=\oline V$. 
\end{lem}

\begin{proof} Let $p$ be a $G$-continuous Hilbert norm on $V$.
Further let $\tilde q$ be a $G$-continuous Hilbert  norm on $U$ and 
$q:=\tilde q|_V$. We have to show that $q\prec p$. 
Let $\pi: \tilde U \to \tilde V$ be the map dual to the inclusion 
$V\to U$. As $\tilde U$ is good we get that $p^* \prec \tilde q^* \circ \pi$. 
It follows that $q\prec p$.\end{proof}

\par We conclude this paragraph with an observation which will be frequently 
used later on. 

\begin{lem}\label{ll3} Let $V_1\subset V_2\subset V_3$ be an inclusion 
chain of Harish-Chandra modules. Suppose that $V_2$ and     
and $V_3/V_1$ are good. Then $V_2/V_1$ is good. 
\end{lem}

\begin{proof}
Let $\oline {V_3}$ be an SF-globalization of $V_3$. 
Let $\oline V_1$, $\oline {V_2}$ be the closures of $V_{1,2}$ in 
$\oline V_3$. 
As $V_2$ is good we have $\oline {V_2} = V_2^\infty$ and thus Lemma
\ref{ll1} implies that
$\oline {V_2}/\oline {V_1} = (V_2/V_1)^\infty$. 
Our second assumption gives $(V_3/V_1)^\infty= \oline{V_3}/ \oline{V_1}$ 
and Lemma \ref{ll2} yields in addition  that 
$\oline{V_2}/ \oline{V_1}= (V_2/V_1)_{\rm max}^\infty$.
\end{proof}

\section{Lower bounds for matrix coefficients}

The objective of this section is to show that 
Harish-Chandra modules are good if and only 
if they feature certain lower bounds for matrix 
coefficients which are uniform in the $K$-type. 

\par As before, given a Harish-Chandra module $V$ we fix a finite dimensional $\cZ(\gf)$-invariant set of generators
$\Xi\subset \tilde V$. We let $\xi_1, \ldots, \xi_k$ be a basis of $\Xi$. 
For $r>0$ we define balls in $G$ by 
$$B_r:=\{ g\in G\mid \|g\|<r\}\, . $$
Set $r_0:=\min\{\|g\|\mid g\in G\}\geq 1$.

\begin{thm}\label{thm=lobo} Let $V$ be a Harish-Chandra module. Then $V$ is good 
if and only if for all $G$-continuous norms $q$ on $V^\infty$
there exists a choice of $\Xi\subset\tilde V$ and constants $c_1, c_2, c_3>0$ 
such that 
$$\Big(\sum_{j=1}^k \int_{B_r} |\xi_j(\pi(g)v)|^2 \ dg\Big)^{1\over 2}  \geq 
{c_2\over (1+|\tau|)^{c_3}}\cdot  q(v)$$ 
for all $\tau \in \hat K$, $v\in V[\tau]$ and $r>\max\{r_0, (1+|\tau|)^{c_1}\}$
\end{thm}
In view of the local Sobolev Lemma this is equivalent to the following pointwise version. 

\begin{thm}\label{thm=lobop} Let $V$ be a Harish-Chandra module. Then $V$ is good 
if and only if for all $G$-continuous norms $q$ on $V^\infty$
there exists a choice of $\Xi$  
and constants $c_1, c_2, c_3>0$ 
such that for all $\tau \in \hat K$ and $v\in V[\tau]$ 
there exist a $g_\tau \in G$ such that $\|g_\tau \|\leq (1+ |\tau|)^{c_1}$
and 
$$\max_{1\leq j\leq k}|\xi_j(\pi(g_\tau)v)|\geq {c_2\over (1+|\tau|)^{c_3}}\cdot  q(v)\, .$$ 
\end{thm}

\begin{proof}\footnote{Throughout this paper we use the convention that 
capital constants $C>0$ might vary from line to line} 
Suppose that $V$ is good. We shall establish the pointwise lower bound 
in Theorem \ref{thm=lobop}.
By assumption there exists an 
$n\in \N$ and $C>0$ such that 
$$|\xi_j(\pi(g)v)|\leq C \cdot \|g\|^n q (v)$$
for all $v\in V^\infty$, $g\in G$ and $1\leq j\leq k$. 
For $N\geq n$ we write $E_N$ for the 
Banach completion of $V^\infty$ with respect to the norm 

$$p^N(v):=\max_{1\leq j\leq k} \sup_{g\in G} {|\xi_j(\pi(g)v)|\over 
  \|g\|^N}\qquad (v\in V)\, .$$ 
We recall that $E_N$ is a Banach module for $G$ (cf. Lemma \ref{l=app}). 

As $V$ is good, we obtain that 

\begin{equation}\label{NN'} V^\infty= E_N^\infty=E_{N'}^\infty\end{equation}
for all $N,N'\geq n$. 
Now fix $N$ and let $N'=N+l >N$.
In view of Proposition \ref{lem=K} 
there exists an $s\in 2\N_0$ and $C>0$ such that

\begin{equation} \label{N,N'} p^N(v)\leq C\cdot p^{N'}_{ s, K}(v)\end{equation}
for all $v\in V^\infty$. 
\par Let us fix $\tau\in \hat K$, $v\in V[\tau]$ and 
$g_\tau \in G$ such that $g\mapsto \max_{1\leq j\leq k}{|\xi_j(\pi(g)v)|\over \|g\|^{N'}}$
becomes maximal at $g_\tau$. 
We then derive from (\ref{N,N'}) that  

$$\max_{1\leq j\leq k} {|\xi_j(\pi(g_\tau)v)|\over \|g_\tau\|^N}\leq C\cdot (1+\|\tau\|)^s
\cdot \max_{1\leq j\leq k}{|\xi_j(\pi(g_\tau)v)|\over \|g_\tau\|^{N+l}}$$
for all $v\in V[\tau]$, i.e.
$$\|g_\tau\|\leq  C\cdot (1+\|\tau\|)^{s\over l}\, .$$
Here $\rho_\kf\in i\tf^*$ is the usual half sum  $\rho_\kf ={1\over 2} \tr \ad_{\kf}$. 

On the other hand (\ref{NN'}) combined with Proposition \ref{lem=K} 
implies likewise that
there exists $C>0$ and $s'>0$ such that 

$$q(v)\leq C\cdot p^{N'}_{s',K}(v)$$
for all $v\in V^\infty$. For $v\in V[\tau]$ we then get 

$$|\xi(\pi(g_\tau)v)|\geq {C\cdot \|g_\tau\|^{N'}\over   
(1+\|\tau\|)^{s'} }\cdot q(v)\, .$$
As $\|g\|\geq 1 $  for all $g\in G$,  the asserted lower bound is established. 

\par Assume now that the lower bound in Theorem \ref{thm=lobo} holds true. 
Let $N>0$ be large enough so that $m_{\xi_j, v}$ is square integrable 
with respect to ${dg\over \|g\|^N}$ and define a Hermitian norm by 
$$p(v)^2:=\sum_{j=1}^k \int_G  |\xi_j(\pi(g)v)|^2 \ {dg\over \|g\|^N}\, . $$
In view of the lower bound we readily obtain that $q\prec p$. On the other hand, 
for large enough $N$ we have that $p\prec q$ and the proof is complete. 
\end{proof}

\section{Minimal principal series representations} \label{mpri}

This section is devoted to the  study of minimal principal 
series representation of $G$ and contains one of our main results. In particular we will show that 
all modules of the minimal principal series are good.

\par Recall the minimal parabolic subgroup $P_{\rm min}=MAN$ of $G$. 
For a finite dimensional $P_{\rm min}$-module $W$ we considered 
the corresponding induced module $V:=I(W)$ with its canonical 
Hilbert globalization $\H:=L^2(W\times_M  K)$. Note that 
$\H^\infty=C^\infty(W\times_M K)$.  
In the sequel 
$\|\cdot\|$ will refer to the $L^2$-norm on $\H$. 
We now state one of the main theorems of this paper: 

\begin{thm} \label{thm=main}Let $V=I(W)$ be a minimal principal series representation 
of $G$ and $\H=L^2(W\times_{P_{\rm min}} K)$ its canonical Hilbert
globalization. Let $\xi_1, \ldots, \xi_k$ be a set of generators
of $V$. 
\begin{enumerate}
\item Then there exists constants, $c_1, c_2, C_1, C_2>0$ such that 
for all $\tau\in \hat K$ and $v_\tau\in V[\tau]$ there exists 
functions $f_{\tau, 1}, \ldots, f_{\tau, k}\in C_c^\infty(G)$ with 
the following properties: 
\begin{enumerate}
\item $\sum_{j=1}^k \Pi(f_{\tau, j})\xi_j = v_\tau$. 
\item $\supp (f_{\tau,j}) \subset \{ g\in G\mid \|g\|<
  C_1(1+\|\tau\|)^{c_1}\}$, for all $1\leq j\leq k$. 
\item $\sum_{j=1}^k \|f_{\tau, j} \|_1 \leq C_2 \cdot \|v_\tau\| 
\cdot (1+\|\tau\|)^{c_2}$, 
where $\|\cdot\|_1$ refers to the norm in $L^1(G)$. 
\end{enumerate}
\item One has $C^\infty(W\times_K W)=V^\infty$ and the surjection  
$$\S(G)^k\mapsto  V^\infty, (f_1, \ldots, f_k)\mapsto \sum_{j=1}^k 
\Pi(f_j)\xi_j$$
admits a continuous linear section $V^\infty \to \S(G)^k$. 
\item $V$ is good. 
\end{enumerate}
\end{thm}
   
\begin{rem} (a) Below we will deduce (ii) from (i). Further 
(iii) follows from (i) as $\widetilde{I(W)}\simeq I(W^*)$. 
In Appendix B we reduce assertion (i)  
to the case of spherical principal series representations and prove it 
in this case. 
\par\nin (b) The constant $c_1$  can be made 
explicite. Certainly it depends on the particular norm 
$\|\cdot\|$ on $G$. Let us fix a specific $K\times K$-invariant norm, namely 
$$\|a\|=\sum_{w\in W} a^{w\rho} \qquad (a\in A)$$
with $\rho\in \af^*$ the Weyl half sum. Then our proof shows that any choice of $c_1$ with  
$${1\over 2} c_1>\dim \af =\rank_\R (G)$$
is  possible. 
\par The constant $c_2>0$ depends only on the growth rate of the $P_{\rm min}$-representation $W$. 
\end{rem}

In view of Casselman's embedding theorem  we can embed every Harish-Chandra
module $V$ into a minimal principal series modules $I(W)$. As $I(W)$ is good 
by Theorem \ref{thm=main} (iii) we thus conclude from Lemma \ref{ll2}:

\begin{cor}\label{cor=mm}  Every Harish-Chandra module $V$ admits a maximal globalization
$V^\infty_{\rm max}$. In particular, $V$ admits a unique minimal and maximal 
$G$-continuous norm with respect to the Sobolev order $\prec$.  
\end{cor}

\medskip  
\noindent {\bf Proof of Theorem \ref{thm=main} (ii).}  Assuming 
Theorem \ref{thm=main}(i) we are going to establish (ii). 
For any $\tau\in \hat K$, let $v_{\tau, 1}, \ldots, v_{\tau, l(\tau)}$
be an orthonormal basis of the $\tau$-isotypical component 
$\H[\tau]=V[\tau]$. Let $v\in V^\infty=\H^\infty$ be a smooth vector, that is 
$v=\sum_{\tau\in  \hat K} \sum_{j=1}^{l(\tau)} c_{\tau, j} v_{\tau,j}\in \H$ and 
for all $N>0$ one has 
\begin{equation} \label{Kest} 
\sum_{\tau\in \hat K} \sum_{j=1}^{l(\tau)} |c_{\tau, j}| (1+\|\tau\|)^N <\infty\, .\end{equation} 
Given $\tau$ and $1\leq j\leq l(\tau)$ we choose $f_{\tau,j, 1}, 
\ldots, f_{\tau_,j, k}$ as in (i),
(a)-(c),  that is $\sum_{i=1}^k \Pi(f_{\tau, j, i} )\xi_i= v_{\tau,  j}$ etc. 
\par For all $1\leq i\leq k$ we set $f_i:=\sum_{\tau\in \hat K}
\sum_{j=1}^{l(\tau)} c_{\tau, j} f_{\tau,j, i}$. 
We first claim that $\int_G |f_i(g)| \cdot \|g\|^r\ dg <\infty$ for all 
$r>0$. In fact, 
\begin{align*} \int_G |f_i(g)| \cdot \|g\|^r\ dg &\leq \sum_{\tau}\sum_j
  |c_{\tau, j}|  
\int_G |f_{\tau,j, i}(g)| \cdot \|g\|^r\ dg \\ 
&\leq \sum_{\tau, j} |c_{\tau, j}| \int_{\{\|g\|\leq C_1 (1+\|\tau\|^{c_1})\}} 
|f_{\tau,j,i}(g)| \cdot \|g\|^r\
  dg\\ 
&\leq \sum_{\tau, j}  C_1^r (1+\|\tau\|)^{r c_1}  |c_{\tau, j}| \int_G |f_{\tau,j, i}(g)|  dg\\ 
&\leq C_1^rC_2  \sum_{\tau, j}  (1+\|\tau\|)^{r c_1+ c_2}  |c_{\tau, j}|  
\end{align*} 
which is finite in view of (\ref{Kest}). 
\par Note that $\sum_{j=1}^k \Pi(f_j)\xi_j = v$. Now for each $1\leq j\leq k$ we choose a $K\times K$-finite test function
$\phi_j$ such that $\Pi(\phi_j) \xi_j=\xi_j$. This allows us to replace $f_j$ 
by $f_j*\phi_j$ and thus we may assume that  $f_j\in L^1(G, \|g\|^r
dg)^\infty$ for all $r>0$. 
Here ``$\infty$'' refers   the smooth vectors of the right regular
representation of $G$ on $L^1(G, \|g\|^r dg)$. 
In view of Remark \ref{rem=r=s} we obtain 
$f_j\in \S(G)$ as to be shown.  
\qed

\section{Reduction steps: extensions, tensoring and induction}

In this section we will show that ``good'' is preserved by induction, tensoring with 
finite dimensional representations and as well by extensions. 
We would like  to emphasize that these results are not new can be found, 
mostly we different proofs,  for instance in \cite{W2},  
Sect. 11.7.

\subsection{Extensions}
 
\begin{lem}\label{lem=3} Let 
$$0\to U\to L \to V\to 0 $$
 be an exact sequence of Harish-Chandra modules. If $U$ and $V$ are good, then $L$ is 
good. 
\end{lem}
\begin{proof} Let $(\pi, \oline L)$ be a smooth Fr\'echet 
globalization of $L$. 
Define a smooth Fr\'echet globalization $(\pi_U, \oline U)$ of $U$ 
by taking the closure of $U$ in $\oline L$. Likewise we define a smooth 
Fr\'echet globalization $(\pi_V, \oline V)$  of $V=L/U$ by 
$\oline V:= \oline L/ \oline U$. 
By assumption we have $\oline U =\Pi_U(\S(G)) U$ and $\oline V=
\Pi_V(\S(G))V$. 
As $0\to \oline U \to \oline L \to \oline V\to 0$ is exact,  we deduce 
that $\Pi(\S(G))L=\oline L$ as vector spaces. Finally the open mapping theorem
implies that  $\Pi(\S(G))L=\oline L$ as topological vector spaces, i.e. $L$ is good. 
\end{proof}

\par As Harish-Chandra modules admit finite composition series we conclude: 

\begin{cor}\label{cor=red} In order to show that all Harish-Chandra modules are good it is 
sufficient to establish that all irreducible Harish-Chandra modules 
are good. 
\end{cor}

\subsection{Tensoring with finite dimensional representations}\label{tensor}
This subsection is devoted to tensoring a Harish-Chandra module 
with a finite dimensional representation.

\par Let $V$ be a Harish-Chandra module and $V^\infty$ its minimal globalization. 
Let 
$(\sigma, W)$ be a finite dimensional representation of $G$. 
Set 

$$\bV:= V\otimes W$$
and note that $\bV$ is a Harish-Chandra module as well. 
It is our goal to show that the minimal globalization 
of $\bV$ is given by $V^\infty \otimes W$. 

\par Let us fix a covariant inner product  $\la\cdot, \cdot\ra$ 
on $W$. Let  
$w_1, \ldots, w_k$ be a corresponding orthonormal basis of $W$. 
With that we define the $C^\infty(G)$-valued $k\times k$-matrix 
$$\mathfrak{S}:=\left(\la \sigma(g)w_i, w_j\ra\right)_{1\leq i, j\leq k}$$
and record the following:

\begin{lem} With the notation introduced above, the following assertions
hold: 
\begin{enumerate} 
\item The map 
$$\S(G)^k \to \S(G)^k, \ {\bf f}=(f_1,\ldots, f_k)\mapsto 
\mathfrak{S} ({\bf f})$$
is a linear isomorphism. 
\item The map 
$$[C_c^\infty (G)]^k \to [C_c^\infty(G)]^k, 
\ {\bf f}=(f_1,\ldots, f_k)\mapsto 
\mathfrak{S} ({\bf f})\, .$$
is a linear isomorphism. 
\end{enumerate}
\end{lem}

\begin{proof} First, we  observe that the determinant 
of $\mathfrak{S}$ is $1$ and hence $\mathfrak{S}$ is invertible. 
Second, all coefficients of  $\mathfrak{S}$ and $\mathfrak{S}^{-1}$
are of moderate growth, i.e. dominated by a power of $\|g\|$. 
Both assertions follow.
\end{proof}

\begin{lem} \label{lem=1}Let $V$ be a Harish-Chandra module 
and $(\sigma, W)$ be a finite 
dimensional representation of $G$. Then 
$$\bV^\infty= V^\infty \otimes W\, .$$ 
\end{lem}

\begin{proof} We denote by $\pi_1=\pi\otimes \sigma$ the tensor 
representation of $G$ on $V^\infty\otimes W$. 
It is sufficient to show that 
$v\otimes w_j$ lies in $\Pi_1(\S(G))\bV$ for all 
$v\in V^\infty$ and $1\leq j\leq k$.

\par Fix $v\in V^\infty$. It is no loss of generality to 
assume that $j=1$. 
By assumption we find $\xi\in V$ and 
$f\in \S(G)$ such that $\Pi(f)\xi=v$. 
\par We use the previous lemma and obtain an ${\bf f}=
(f_1, \ldots, f_k)\in \S(G)^k$ such that 
$$\mathfrak{S}^t({\bf f})= (f, 0, \ldots, 0)\, .$$ 
We claim that 

$$ \sum_{j=1}^k \tilde \Pi(f_j) (\xi\otimes w_j)= v\otimes w_1\, .$$
In fact, contracting the left hand side with $w_i^*=\la \cdot, w_i\ra$ we get 
that 

\begin{align*} 
(\mathrm{id}\otimes w_i^*)\left(\sum_{j=1}^k \Pi_1(f_j) (\xi\otimes w_j)
\right)&= \sum_{j=1}^k \int_G f_j (g) \la \sigma(g)w_j, w_i\ra 
\pi(g)\xi \ dg \\
&= \delta_{1i} \cdot \int_G f(g) \pi(g)\xi\ dg = \delta_{1i}\cdot v 
\end{align*} 
and the proof is complete.
\end{proof}
 
\begin{prop}\label{prop=bi} Let $V$ be a good Harish-Chandra module 
and $(\sigma, W)$ be a finite 
dimensional representation of $G$. Then $\bV= V\otimes W$ 
is good.
\end{prop}

\begin{proof} It is easy to see that maximal and minimal Sobolev norms
(with respect to $\prec$) on $V$ induce maximal and minimal 
Sobolev norms on $\bV$. The assertion follows. 
\end{proof}

\subsection{Induction}
Let $P\supset P_{\rm min}$ be a parabolic subgroup with 
Langlands decomposition 
$$P=N_P A_P M_P\, .$$
Note that $A_P<A$,  $M_P A_P = Z_G(A_P)$ and 
$N=N_P \rtimes (M_P \cap N)$. 
For computational purposes it is useful to recall that 
parabolics $P$ above $P_{\rm min}$ are parameterized by subsets
$F$ of the simple roots $\Pi$ in $\Sigma(\af, \nf)$. We then often write 
$P_F$ instead of $P$, $A_F$ instead of $A_P$ etc. 
The correspondence $F\leftrightarrow P_F$ is such that 
$$A_F=\{ a\in A\mid (\forall \alpha\in F)\  a^\alpha=1\}\, .$$
We make an emphasis on the two extreme cases for $F$, namely: 
$P_\emptyset=P_{\rm min}$ and $P_\Pi=G$. 

\par In the sequel we write $\af_P$, $\nf_P$  for the Lie algebras
of $A_P$ and $N_P$ and denote by $\rho_P\in \af_P^*$ the usual half sum. 
Note that $K_P:=K\cap M_P$ is a maximal compact subgroup of $M_P$.
Let $V_\sigma$ be a Harish-Chandra module for $M_P$ and 
$(\sigma, V_\sigma^\infty)$ its minimal SF-globalization. 

\par For $\lambda\in (\af_P)_\C^*$ we define as before 
the smooth principal series with parameter $(\sigma, \lambda) $ by 
\begin{align*} E_{\sigma, \lambda}=\{ f\in 
C^\infty(G, V_\sigma^\infty))\mid & (\forall\  nam\in 
P\ \forall\  g\in G)\\
& f(nam g) = a^{\rho_P+\lambda} \sigma(m) f(g)\}\, .\end{align*}
and representation $\pi_{\sigma, \lambda}$ by right translations
in the arguments of functions in $E_{\sigma, \lambda}$.

In this context we record: 

\begin{prop} \label{p=ind} Let $P\supseteq P_{\rm min}$ be a parabolic subgroup 
with Langlands decomposition $P= N_P A_P M_P$. Let $V_\sigma$ be an
irreducible good Harish-Chandra module for $M_P$. 
Then for all 
$\lambda\in (\af_P)_\C^*$ the induced Harish-Chandra module 
$V_{\sigma, \lambda}$ is good. In particular, $V_{\sigma, \lambda}^\infty=
E_{\sigma, \lambda}$. 
\end{prop} 

\begin{proof} As $\tilde V_{\sigma, \lambda}\simeq V_{\tilde \sigma, -\lambda}$ 
it is sufficient to show that 
$V_{\sigma, \lambda}$ is good. 

\par In the first step we will show that $E_{\sigma, \lambda}$
is the maximal globalization of $V_{\sigma, \lambda}$. 
To begin with let  $N^P:=M_P \cap N$ and $A^P:= M_P \cap A$. Then 
$Q:= N^P A^P M$ is a minimal parabolic subgroup of $M_P$. As 
$V_\sigma$ is irreducible, we find an embedding of $V_\sigma$ into 
a minimal principal series module of $M_P$, 
say $I_\sigma$:
$$I_\sigma =\Ind_Q^{M_P} ({\bf 1} \otimes(\mu+\rho^P)\otimes \gamma) $$ 
with $\mu \in (\af^P)_\C^*$ and $(\gamma, U_\gamma)$ an
irreducible representation of $M$. 
Then 
$L^2(U_\gamma\times_M K_P)$ is a Hilbert globalization of 
$I_\sigma$ and we denote by $\H_\sigma$ the closure of $V_\sigma$ in 
$L^2(U_\gamma\times_M K_P)$. As $V_\sigma$ is good, it follows that 
$V_\sigma^\infty=\H_\sigma^\infty$.
With $\H_\sigma$ we obtain a Hilbert model for $V_{\sigma, \lambda}$ namely
$\H_{\sigma, \lambda}=L^2(\H_\sigma\times_{K_P} K)$. Notice that 
the smooth vectors of $\H_{\sigma, \lambda}$ coincide with  
$E_{\sigma, \lambda}$.

\par We proceed with double induction (see \cite{K}, Ch. VII, \textsection  2 (4))  and obtain an
embedding of $V_{\sigma, \lambda}$ into the minimal principal series 
representation $\bV:= V_{\gamma, \mu+\lambda}$. 
Let us endow  $\bV$ with the Hilbert structure 
induced from the compact model $\bH=L^2(U_\gamma\times_M K)$.  
Observe that the embedding $\H_{\sigma, \lambda}$ to 
$\bH$ is isometric. 
As $\bV$ is good we get that 
the maximal globalization of $V_{\sigma, \lambda}$ is the closure of 
$V_{\sigma, \lambda}$ in $\bV^\infty$. From our discussion it follows that 
this closure is  $\H_{\sigma, \lambda}^\infty=E_{\sigma, \lambda}$. 

\par To conclude the proof we need to show that 
$E_{\sigma, \lambda}$ coincides with the minimal globalization 
of $V_{\sigma, \lambda}$ as well. We proceed dually: start from the 
realization of $V_\sigma$ as a quotient of a minimal principal 
series $J_\sigma$ of $M_P$ etc. As before we will end up with a
Hilbert model $\hat H_{\sigma, \lambda}$ for $V_{\sigma, \lambda}$
with  $\hat H_{\sigma, \lambda}^\infty= E_{\sigma, \lambda}$
and an orthogonal projection of some Hilbert globalization 
$\hat\bH$ of some good tensor product module onto 
$\hat H_{\sigma, \lambda}$. Hence Lemma \ref{ll1} implies that 
$E_{\sigma, \lambda}$ equals the minimal globalization. 
\end{proof}

\section{Reduction steps II: deformation theory and discrete series}

We already know that every irreducible Harish-Chandra modules $V$ can be written
as a quotient $U/H$ where $U$ is good. Suppose that $H$ is in fact a kernel 
of an intertwiner $I: U\to W$ 
with $W$ good. Suppose in addition that  we can deform $I: U\to W$  holomorphically 
(as to be made precise below). Then, provided $U$ and $W$ 
are good we will show that 
$\im I\simeq U/H$ is good. 
In view of the Langlands-classification,  
the assertion that every Harish-Chandra module is good 
then reduces to the case of discrete series representations.

\par This section is organized as follows: we recall the holomorphic deformation theory
of Casselman (see \cite{Cas}, Sect. 9) in a slightly modified form. Then we prove that 
discrete series are good, and, finally, prove the Casselman-Wallach globalization theorem.

\subsection{Deformation theory} 

For a complex manifold $D$ and a Harish-Chandra module $U$ we write 
$\O(D,U)$ for the space of maps $f: D\to U$ such that for all 
$\xi\in \tilde U$ the contraction $\xi\circ f$ is holomorphic. 
Henceforth we will use $D$ exclusively for the open unit disc. 

\par By a holomorphic family of Harish-Chandra modules (parameterized by $D$) 
we understand a family  of Harish-Chandra modules 
$(U_s)_{s\in D}$ such that:

\begin{enumerate}
\item For all $s\in D$ one has $U_s = U_0=:U$ as $K$-modules. 
\item  For all $X\in \gf$, $v\in U$ and $\xi\in \tilde U$ the map 
$s\mapsto \xi (X_s\cdot v)$ is holomorphic. Here we use $X_s$ for the 
action of $X$ in $U_s$. 
\end{enumerate}

Given a holomorphic family $(U_s)_{s\in D}$ we form 
$\U:=\O(D, U)$ and endow it with the following $(\gf,K)$-structure: 
for $X\in \gf$ and $f\in \U$ we set 
$$(X\cdot f)(s):=X_s\cdot f(s)\, .$$  
We emphasize that the algebra multiplication of $\O(D)$ on $\U$ commutes with the
$(\gf, K)$-action.

\par Of particular  interest are the Harish-Chandra modules 
${\bf U}_k:=\U/s^k\U$ for $k\in \N$. To get a feeling 
for this objects let us discuss a few examples for small 
$k$. 

\begin{ex}(a) For $k=1$ the constant term map 
$${\bf U}_1\to U, \ \ f + s\U\mapsto f(0)$$
is an isomorphism of $(\gf, K)$-modules. 
\par (b) For $k=2$ we observe that the map 

$${\bf U}_2 \to U\oplus U, \ \ f+ s^2 \U \mapsto (f(0), f'(0))$$
provides an isomorphism of $K$-modules. The resulting 
$\gf$-action on the right hand side is twisted 
and given by 

$$X\cdot (u_1,u_2)=(Xu_1, Xu_2 + X'u_1)$$
where 
$$X'u:={d\over ds}\Big|_{s=0} X_s\cdot u\, .$$ 
Let us remark that $X'=0$ for all $X\in \kf$. 
\par We notice that ${\bf U}_2$ features the submodule $s\U/s^2\U$ which corresponds
to $\{0\} \oplus U$ in the above trivialization.  
The corresponding quotient  
$(\U/s^2\U)/ (s\U/s^2\U) $ identifies with $U\oplus U/ \{0\}\oplus U
\simeq U$. In particular $\U/s^2\U$ is good if $U$ is good by the 
extension Lemma \ref{lem=3}.  
\end{ex}

From the previous discussion it follows that ${\bf U}_k$ 
is good for all $k\in \N_0$ provided that $U$ is good.

\par Let now $W$ be another Harish-Chandra module and $\W$ a 
holomorphic deformation of $W$ as above. By a morphism 
$\I: \U\to \W$ we understand a family of $(\gf,K)$-maps 
$I_s: U_s \to W_s$ such that for all $u\in U$ and $\xi\in W^*$
the assignments $s\mapsto \xi(I_s(u))$ are holomorphic. 
Let us write $I$ for $I_0$ set $I':={d\over ds}\Big|_{s=0} I_s$ etc. 
We set $H:=\ker I$.

\par We now make two additional assumptions on our holomorphic 
family of intertwiners: 

\begin{itemize} \label{i=1}
\item $I_s$ is invertible for all $s\neq 0$. 
\item There exists a $k\in \N_0$ such that 
$J(s):=s^k I_s^{-1}$ is holomorphic on $D$. 
\end{itemize}

If these conditions are satisfied, then we call $I: U\to W$ 
{\it holomorphically deformable}.   

\par For all $m\in \N$ we write  ${\bf I}_m: {\bf U}_m \to {\bf W}_m$ 
for the intertwiner induced by $\I$. Likewise we define ${\bf J}_m$.

\begin{ex} In order to get a feeling for the intertwiners ${\bf I}_m$
let us consider the example ${\bf I}_2: {\bf U}_2 \to {\bf W}_2$. 
In trivializing coordinates  this map is given by 
$${\bf I}_2(u_1,u_2)= (I(u_1), I(u_2)+ I'(u_1))\, .$$
\end{ex}

We set ${\bf H}_m:= \ker {\bf I}_m\subset {\bf U}_m$. 
For $m<n$ we view ${\bf U}_m$ as a $K$-submodule of ${\bf U}_n$
via the inclusion map 
$${\bf U}_m \to {\bf U}_n , \ \ f + s^m \U \mapsto 
\sum_{j=0}^{m-1} {f^{(j)}(0)\over j!} s^j + s^n \U\, .$$
We write $p_{n,m}: {\bf U}_n \to {\bf U}_m$ for the reverting 
projection (which are $(\gf, K)$-morphisms). 

\par The following Lemma is related to an observation of Casselman, see 
\cite{Cas}, Prop. 9.3. 

\begin{lem} Suppose that $k\in \N_0$ is minimal such that $J(s)$ is holomorphic. Then 
the morphism 
$$ {\bf I}_{2k}|_{{\bf H}_k + s^k\U/ s^{2k}\U}: 
{\bf H}_k + s^k\U/ s^{2k}\U\to  s^k\W/ s^{2k}\W$$
is onto. Moreover, its kernel is given by $s^k {\bf H}_k \subset
s^k\U/ s^{2k}\U$.
\end{lem}

\begin{proof} Clearly, ${\bf I}_{2k}({\bf H}_k + s^k\U/ s^{2k}\U)\subset s^k \W/ s^{2k}\W $ and hence the map is defined. 
Let us check that it is 
onto. Let $[w]\in s^k\W/ s^{2k} \W$ and $w\in s^k \W$ be a representative. 
Note that $\I^{-1}|_{s^k \W}: s^k \W \to \U$ is defined. 
Set $u:= \I^{-1}(w)$ and write $[u]$ for its equivalence class 
in ${\bf U}_{2k}$. Then $u\in {\bf H}_k + s^k\U/ s^{2k}\U$ and  
${\bf I}_{2k}([u])=[w]$ which shows that the map is onto.  

\par A simple verification shows that $s^k {\bf H}_k$ lies in the kernel. 
Further, the first part of the proof shows that $p_{2k, k}\circ {\bf I_{2k}}^{-1}: 
 s^k\W/ s^{2k}\W \to {\bf H}_k$ is an injection. Thus  $s^k {\bf H}_k$ is 
the kernel. 
\end{proof}

If we set $V_3:={\bf H}_k + s^k\U/ s^{2k}\U$, $V_2:= s^k\U/ s^{2k}\U$ and 
$V_1:=s^k {\bf H}_k$, the previous Lemma implies 
an inclusion chain 

$$V_1\subset V_2\subset V_3$$
with  
$$V_2/V_1\simeq {\bf U}_k /{\bf H}_k, \qquad V_2\simeq {\bf U}_k 
\qquad\hbox{and}\qquad  V_3/V_1\simeq {\bf W}_k\, .$$ 
Hence in combination with the squeezing Lemma \ref{ll3}
we obtain that ${\bf U}_k/ {\bf H}_k$ is good if $U$ and $W$ are good. 
\par We wish to show that $U/H$ is good. 
Write $H_{k,1}:=p_{k,1}({\bf H}_k)$ for the projection of ${\bf H}_k$ to 
${\bf U}_1\simeq U$. Note that $H_{k,1}$ is a submodule of $H$. 
We arrive at the exact sequence 

$$0\to U/H\simeq s^{k-1}U/ s^{k-1} H \to {\bf U}_k/ {\bf H}_k \to U/H_{k,1}\to 0\, .$$
But $U/H$ is a quotient of $U/H_{k,1}$. Thus putting an SF-topology on 
$U$ we get one on $H$, ${\bf U}_k$, ${\bf H}_k$,   ${\bf U}_k/ {\bf H}_k$
and  $U/H_{k,1}$. As a result the induced topology on $U/H$ is both a 
sub and a quotient of the good topology on ${\bf U}_k/ {\bf H}_k$. 
Hence $U/H$ is good.   

\par We summarize our discussion. 

\begin{prop}\label{p=def} Suppose that $I: U\to W$ is an intertwiner of 
good Harish-Chandra modules which allows
holomorphic deformations $\I: \U \to \W$. 
Then $\im I$ is good.
\end{prop}

\subsection{Discrete series}

The objective of this subsection is to show that every Harish-Chandra module belonging to 
the discrete series is good. 

\par Let $Z<G$ be the center of $G$. Throughout this section $V$ shall denote a unitarizable irreducible Harish-Chandra module, i.e. 
there exists a unitary irreducible globalization  $(\pi, \H)$ of $V$. 
We say that $V$ is {\it square integrable} or {\it belongs to the discrete series} provided for all 
$v\in V$ and $\xi\in \tilde V$ one has 

$$\int_{G/Z} |m_{\xi, v}(g)|^2 \ d(gZ)<\infty\, .$$
In this situation, there exists a constant $d(\pi)$, the formal degree, such that for every unitary 
norm $p$ on $V$ one has

\begin{equation}\label{sor}
 {1\over d(\pi)} p(v)^2 p^*(\xi)^2  = \int_{G/Z} |m_{\xi, v}(g)|^2 \ d(gZ)\qquad (v\in V, \xi\in \tilde V)\, .
\end{equation}
\begin{prop}\label{DS} Let $V$ be  a Harish-Chandra module of the discrete series. Then $V$ is good. 
\end{prop}

\begin{proof} It is no loss of generality to assume that the center 
$Z<G$ is compact. Choose a minimal principal series $U:=V_{\sigma, \lambda}$ such that 
$V$ embeds into $U$. Let $\xi\in \tilde U$ such that $\xi|_V\neq 0$. Then there exists
an $s_0>0$ such that all matrix coefficients $m_{\xi, v}$ , $v\in U$, belong to 
$L^2(G, \|g\|^{-{s_0}} dg)$. For every $s\in \C$ with $\Re s >s_0$ 
we define a continuous Hermitian form on $U^\infty$ by setting 

$$(v, w)_s:= \int_G m_{\xi, v}(g) \oline{m_{\xi, w}(g)} \|g\|^{-s}\ dg \, .$$
Let us write $B(U^\infty, U^\infty)$ for the topological  vector space
of continuous Hermitian forms on $U^\infty$. We obtain a holomorphic map 
$$\{s\in \C\mid \Re s> s_0\}\to B(U^\infty, U^\infty), \ \ s\mapsto (\cdot, \cdot)_s$$
and in Appendix B we  show that 
it admits a meromorphic continuation to the 
complex plane. 

\par Let $(\cdot, \cdot)$ be the constant part of $s\mapsto (\cdot,\cdot)_s$ at $s=0$. 
Note that for  $v, w\in V$ we have 
$$(v,w)=\int_G m_{\xi, v}(g) \oline{m_{\xi, w}(g)}\ dg\, . $$ 
As, on the other hand $(\cdot, \cdot) \in B(U^\infty, U^\infty)$ and $U^\infty$ is good, we conclude 
from (\ref{sor}) and Lemma \ref{ll2} that the unitary norm $p$ on $V$ is a maximal norm on $V$ w.r.t. $\prec$. 
As $V$ is unitary, $p$ is also minimal. 
\end{proof}

\subsection{Proof of the Casselman-Wallach-Theorem}

\begin{thm} All Harish-Chandra modules are good.
\end{thm} 

\begin{proof} Let $V$ be a Harish-Chandra module. We have to show that 
$V$ is good. In view of Corollary \ref{cor=red}, we may assume that 
$V$ is irreducible. 
Next we use Langland's classification (see \cite{K},Ch. VIII, Th. 8.54) and 
combine it with 
our Propositions on deformation \ref{p=def} and induction 
\ref{p=ind}. This reduces to the case 
where $V$ is tempered. However, the case of tempered readily reduces to 
square integrable (\cite{W}, Ch. 5, Prop. 5.2.5). The case of square integrable Harish-Chandra modules 
was established in Proposition \ref{DS}.
\end{proof}

\subsection{Discussion}

In the introduction we phrased the Casselman-Wallach in several different
ways. One way was the equivalence of categories $\HC$ and $\SAF$ or,
equivalently,  that there is only one Sobolev-equivalence class of $G$-continuous norms on a Harish-Chandra module.  

\par The objective of this subsection is to show that the equivalence 
of categories $\HC \simeq \SAF$ can be slightly refined. 

\par By a {\it marking} of a Fr\'echet space $E$ we shall understand 
an increasing family $(p_n)_{n\in \N}$ of semi-norms which define the 
topology on $E$. Pairs $(E, (p_n)_n)$ will henceforth be called 
marked Fr\'echet spaces -- in the literature one also finds 
the notion of {\it graded} Fr\'echet space (see \cite{Ha}). 

\par By a morphism of marked Fr\'echet spaces $(E, (p_n)_n)$ and $(F,(q_n)_n)$ 
we understand a linear map $T:E\to F$ with the following property: there 
exists a $k\in \N_0$ such that for all $n\in \N$ there exists $C_n>0$ with: 

$$q_n(T(x))\leq C_n p_{n+k} (x) \qquad (x\in E)\, .$$

\par In this sense we obtain the additive category of marked Fr\'echet spaces, say 
$\F_{\rm mark}$. 

\par For an $F$-representation $(\pi, E)$  we are automatically 
led to the notion of a $G$-continuous marking. 
Let us define now $\SAF_{\rm mark}\subset \F_{\rm mark}$ to be the sub-category of smooth admissible 
$F$-representation with respect to a $G$-continuous marking. 
The refined Casselman-Wallach theorem asserts that 

\begin{equation} \label{CWref} \HC\simeq \SAF_{\rm mark}\, .\end{equation}
Note that this immediate from the fact that $V^\infty$ is a quotient of 
$\S(G)^k$ for some $k\in \N$.

\section{Applications}

\subsection{Lifting $(\gf,K)$-morphisms}

Let $(\pi, E)$ be a representation of $G$ on a complete topological vector
space $E$. Let us call  $(\pi, E)$ an {\it $\S(G)$-representation} if 
the natural action of $C_c^\infty(G)$ on $E$ extends to a separately 
continuous action of $\S(G)$ on $E$. 
Some typical examples we have in mind are smooth 
functions of  moderate growth on certain homogeneous spaces. 
Let us mention a few. 

\begin{ex} (a) Let $\Gamma<G$ be a lattice, that is a discrete subgroup 
with cofinite volume. Reduction theory (Siegel sets) allows 
us to control ``infinity'' of the quotient $Y:=\Gamma\bs G$ and 
leads to a natural notion of moderate growth. 
For every $\alpha>0$ there is a natural SF-module 
$C_\alpha^\infty(Y)$ 
of smooth functions on $Y$ with growth rate 
at most $\alpha$. 
The smooth functions
of moderate  growth $C_{\rm mod}^\infty (Y)=\lim_{\alpha\to \infty}
C_\alpha^\infty(Y)$ become a complete  inductive limit 
of the SF-spaces $C_\alpha^\infty(Y)$. Hence $\S(G)$ acts on  
$C_{\rm  mod}^\infty (Y)$. 
\par The space of $K$ and $\cZ(\gf)$-finite 
elements in  $C_{\rm mod}^\infty (Y)$ is referred to as the space of 
automorphic forms on $Y$. 
\par\noindent (b) Let $H<G$ be a symmetric subgroup, i.e. an open 
subgroup of the fixed point set of an involutive automorphism 
of $G$. We refer to $X:=H\bs G$ as a semisimple symmetric space. 
The Cartan-decomposition of $X$ allows us to control growth on $X$ 
and yields natural SF-modules $C_\alpha^\infty(X)$ of smooth 
functions with growth rate at most $\alpha$. As before one obtains 
$C_{\rm mod}^\infty (X)=\lim_{\alpha\to \infty} C_\alpha^\infty(X)$ a natural 
complete $\S(G)$-module of functions with moderate growth. 
\end{ex}

\par If $(\pi_1, E_1)$, $(\pi_2, E_2)$ are two representations, 
then we denote by\linebreak  $\Hom_G(E_1, E_2)$ for the space of continuous $G$-equivariant 
linear maps from $E_1$ to $E_2$. 

\begin{prop}\label{p=ext} Let $V$ be a Harish-Chandra module and $V^\infty$ 
its unique $SF$-globalization. Then for any smooth  $\S(G)$-representation 
$(\pi,E)$ of $G$ the linear map 
$$\Hom_G(V^\infty, E) \to \Hom_{(\gf, K)}(V, E^{K-{\rm fin}}), 
\ \ T^\infty\mapsto T:=T^\infty|_{V}$$
is a linear isomorphism. 
\end{prop}

\begin{proof} It is clear that the map is injective. 
To show that the map is onto let us write $\lambda$, resp. $\Lambda$,  for the representation of
$G$, resp. $\S(G)$, on $V^\infty$. Let $v\in V^\infty$. Then we find $f\in \S(G)$ 
such that $v=\Lambda(f)w$ for some $w\in V$. 
We claim that 
$$T^\infty(v):=\Pi(f)T(w)$$ 
defines a linear operator. 
In order to show that this definition makes sense we have to show that 
$T^\infty(v)=0$ if $\Lambda(f)w=0$. 
Let $\xi \in (E^*)^{K-{\rm fin}}$ and 
$\mu:=\xi\circ T\in \tilde V$. We consider 
two distributions on $G$, namely 

$$\Theta_1(\phi):= \xi (\Pi(\phi)T(w))\quad \hbox{and}\quad 
\Theta_2(\phi):= \mu(\Lambda(\phi)w) \qquad (\phi\in C_c^\infty(G))\, .$$
We claim that $\Theta_1=\Theta_2$. In fact,  both distributions are 
$\cZ(\gf)-$ and $K\times K$-finite. Hence they are represented 
by analytic functions on $G$ and thus uniquely 
determined by their derivatives on $K$.  The claim follows. 

\par It remains to show that $T$ is continuous. We recall the construction
of the minimal SF-globalization of $V$, namely $V^\infty =\S(G)^k/ 
\S(G)_{\bf v}$. 
As the action of $\S(G)$ on $E$ is separately continuous, the
continuity of $T^\infty$ follows. 
\end{proof}

\subsection{Automatic continuity}

For a Harish-Chandra module $V$ we denoted by $V^*$ 
its algebraic dual. Note that $V^*$ 
is naturally a module for $\gf$. 

\par If $\hf<\gf$ is a subalgebra, then we write 
$(V^*)^\hf$, resp.  $(V^*)^{\hf-{\rm fin}}$, for the 
space of $\hf$-fixed, resp. $\hf$-finite, algebraic linear 
functionals on $V$. 

\par We call a subalgebra $\hf<\gf$ a {\it (strong) automatic continuity}
subalgebra ((S)AC-subalgebra for short) if for all Harish-Chandra 
modules $V$ one has 
$$(V^*)^\hf\subset (V^\infty)^* \quad \hbox{resp.} \quad 
(V^*)^{\hf-{\rm fin}}\subset  (V^\infty)^*\, .$$

\begin{problem} (a) Is it true that $\hf$ is AC if and only if  
$\la \exp\hf \ra <G$ has an open orbit on $G/P_{\rm min}$ ?
\par\noindent (b) Is it true that 
$\hf$ is SAC if $[\hf,\hf]$ is AC ?
\end{problem}

The following examples of (S)AC-subalgebras are known: 

\begin{itemize}

\item $\nf$, the Lie algebra of an Iwasawa $N$-subgroup, is AC and 
$\af +\nf$, the Lie algebra of an Iwasawa $AN$-subgroup, is 
SAC.  (Casselman). 
\item Symmetric subalgebras, i.e. fixed point sets of involutive 
automorphisms of $\gf$, are AC (Brylinski, Delorme, van den Ban; cf.  
\cite{BrD}, \cite{BD}). 
\end{itemize}

Here we only wish to discuss Casselman's result. We recall
the definition of the Casselman-Jacquet module 
$j(V)=\bigcup_{k\in \N_0} (V/\nf ^k V)^*$ and note that 
$j(V)= (V^*)^{\af +\nf-{\rm fin}}$.

\begin{thm}{\rm (Casselman)}
Let $\nf$ be the Lie algebra of an Iwasawa $N$-subgroup
of $G$ and $\af+\nf$ the Lie algebra of an Iwasawa $AN$-subgroup. 
Then $\nf$ is an AC  and $\af+\nf$ is SAC. In particular, 
for all Harish-Chandra modules $V$ one has $j(V)\subset (V^\infty)^*$. 
\end{thm}

\begin{proof} Let us emphasize that the proof needs only results up to 
Section \ref{mpri}, i.e. that minimal principal series are good. 
\par We first prove that $\af+\nf$ is SAC. Let 
$V$ be a Harish-Chandra module and $0\neq \lambda \in j(V)$. 
By definition there exists a $k\in \N$ such that 
$\lambda \in (V/\nf^kV)^*$.  
Write $(\sigma, U_\sigma)$ for the finite dimensional representation 
of $P_{\rm min}$ on $V/ \nf^k V$ and denote by $I_\sigma$ the corresponding 
induced Harish-Chandra module. Note that $I_\sigma^\infty=C^\infty
(U_\sigma \times_{P_{\rm min}}G)$.
\par Applying Frobenius reciprocity to the identity morphism 
$V/\nf^k V \to U$ yields a non-trivial $(\gf, K)$-morphism
$T: V\to I_\sigma$ (cf. \cite{W}, 4.2.2). 
Now $T$ lifts to a continuous $G$-map $T^\infty: V^\infty\to I_\sigma^\infty$
by Proposition \ref{p=ext}.  
If ${\rm ev}: I_\sigma^\infty\to U_\sigma$ denotes the evaluation map 
at the identity, then $\lambda^\infty:=  \lambda\circ {\rm ev}\circ 
T^\infty$ provides a continuous extension of $\lambda$ to $V^\infty$. 

\par The assertion that $\nf $ is AC follows from the fact that 
the space of $(V^*)^{\nf}$ is finite dimensional (Casselman), 
and in particular $\af$-finite. 
\end{proof}

\subsection{Lifting of holomorphic families of $(\gf, K)$-maps}

We wish to give a version of lifting (cf. Proposition \ref{p=ext})
which depends holomorphically on parameters. 

To begin with we need a generlaization of 
Theorem \ref{th=1} and Theorem \ref{th=2} for 
principal series representations which are induced from an arbitrary parabolic 
subgroup. 

\par Let $P=N_P A_P M_P$ be a parabolic above $P_{\rm min}$. 
We fix an SAF-representation $(\sigma, V_\sigma^\infty)$ 
of $M_P$ and write $V_\sigma$ for the corresponding Harish-Chandra
module. 

\par As $K$-modules we identify all $V_{\sigma,\lambda}$ with 
$V:=\C[V_\sigma \times_{K_P} K ]$. Note that 
$V_\sigma$ is a $K_P$-quotient of some $\C[K_P]^m$, $m\in \N$. 
Double induction gives an identification of 
$V$ as a $K$-quotient  of $\C[K]^m$. 
Note that $C^\infty (K)^m$ induces the unique SF-topology on $V^\infty$. 
For each $\tau$ we write  $\chi_\tau$  for its character 
and $\delta_{\sigma, \tau}$ for the orthogonal projection 
of $\underbrace{(\chi_\tau, \ldots, \chi_\tau)}_{m-times}$ 
to $V[\tau]$, the $\tau$-isotypical part of $V$.

\begin{thm}\label{th=1a} Let $P=N_PA_PM_P$ be a parabolic subgroup and 
$V_\sigma$ an irreducible unitarizable Harish-Chandra module for 
$M_P$. Let $Q\subset (\af_P)_\C^*$ 
be a compact subset and $N>0$. 
Then there exists $\xi\in \C[V_\sigma \times_{K_P} K]$ and 
constants $c_1,c_2>0$ 
such that for all $\tau\in \hat K$, $\lambda\in Q$,   there exists 
$a_\tau\in A$, independent 
of $\lambda$,  with $\|a_\tau\|\leq (1+|\tau|)^{c_1}$ 
and numbers $b_\sigma(\lambda,\tau)\in \C$  
such that 
$$\|[\pi_{\sigma, \lambda}(a_\tau) \xi]_\tau  - 
b_\sigma(\lambda,\tau) \delta_{\sigma, 
\tau}\| \leq {1\over (|\tau|+1)^{N+c_2}}$$ 
and 
$$|b_\sigma(\lambda,\tau)|\geq {1\over (|\tau|+ 1)^{c_2}}\, .$$
Here $\|\cdot\|$ refers to the continuous norm on $V$ induced 
by the realization of $V$ as a quotient of $C[K]^m\subset L^2(K)^m$. 
\end{thm} 

\begin{proof} Let us first discuss the case where $P=P_{\rm min}$ and 
$\sigma$ is finite dimensional. With the reduction steps explained 
in the beginning of the next section this becomes 
a simple modification 
of Theorem \ref{th=1}. 
\par As for the general case we identify $V_\sigma$ as a quotient 
of a minimal principal series for $M_P$.
Using double induction 
we can write the $V_{\sigma, \lambda}$'s
consistently as quotients of such minimal principal series. The 
assertion follows. 
\end{proof}

As a consequence we get an extension of Theorem \ref{th=2}. 

\begin{thm}  \label{th=2a} Let $Q\subset (\af_P)_\C^*$ be a compact 
subset. Then there exist a continuous map 
$$Q\times C^\infty(V_\sigma^\infty \times_{K_P}K) 
\to \S(G), \ \ (\lambda, v)\mapsto 
f(\lambda, v)$$
which is holomorphic in the first variable, linear in the second 
and such that 
$$\Pi_{\sigma, \lambda}(f(\lambda, v))\xi = v\, .$$ 
\end{thm}
 
As a Corollary to Theorem \ref{th=2a} we obtain 
the holomorphic lifting result: 

\begin{thm} \label{th=3} Let $(\pi, E)$ be a Banach representation of $G$. Within the notation of 
Theorem \ref{th=2a} let $\Omega\subset (\af_P)_\C^*$ be an open set and 
$$T: \Omega\times \C[V_\sigma\times_{K_P} K] \to E^\infty $$
a map  such that: 
\begin{itemize} 
\item For every $v\in \C[V_\sigma\times_{K_P} K]$ the assignment $\Omega\ni \lambda \mapsto T(\lambda, v)\in E^\infty$
is holomorphic. 
\item For every $\lambda$ the assignment 
$$V_{\sigma, \lambda}= \C[V_\sigma\times_{K_P} K] \to E^\infty , \ \ v\mapsto T(\lambda, v)$$
is a $(\gf, K)$-map. 
\end{itemize}
Then $T$ admits a holomorphic extension to  a map 
$$T^\infty: \Omega\times C^\infty(V_\sigma^\infty \times_{K_P} K) \to E^\infty\, .$$ 
\end{thm}

\begin{proof} It is no loss of generality to assume that $\Omega$ is relatively compact. 
Within the notation of Theorem \ref{th=2a} we define 
$$T^\infty(\lambda, v):= \Pi(f(\lambda, v)) T(\lambda, \xi)\, .$$
\end{proof}

\begin{rem} \label{rem=Eis} {\rm (Application to Eisenstein series)}
Let $\Gamma<G$ be a lattice and $Y:=\Gamma\bs G$. Let  
$$T: \Omega\times \C[V_\sigma\times_{K_P} K] \to C_{\rm mod}^\infty(Y)$$
be a map  which satisfies the conditions in Theorem \ref{th=3}. Then, basic automorphic theory 
implies  for all relatively compact 
$\Omega\subset (\af_P)_\C^*$ the existence of a growth index $\alpha$ such that $\im T \subset L_\alpha^2(Y)$. In particular Theorem \ref{th=3} is applicable. 
\par A typical application is as follows. Let us consider
Eisenstein series attached to the lattice $\Gamma<G$.
We assume that $P=M_P A_P N_P $ is cuspidal and set $L:= M_P \cap K$.
Fix a finite dimensional unitary representation $(\sigma, U)$ of $L$
and a $\Gamma\cap M_P$-invariant $L^2$-section $\psi$ of the
vector bundle $ (\Gamma\cap M_P)\bs M_P \times_L U \to (\Gamma\cap M_P)  \bs M_P$.
Suppose that all contractions $\la \psi, u\ra \in L^2(\Gamma\cap M_P \bs M_P)$, $u\in U$, 
generate a Harish-Chandra module for $(\m_P, L)$.  Let $f$ be a smooth section of the 
$K$-equivariant vector bundle $U \times_L K \to L\bs K$. 
Then one defines Eisenstein series

$$E(\lambda,\psi, f) (\Gamma g) := \sum_{\gamma\in \Gamma\cap P \bs \Gamma}
\tilde a(\gamma g)^\lambda \la \psi (\tilde m (\gamma g)), f(\tilde k(g))\ra$$
where $g = \tilde n(g) \tilde a(g) \tilde m(g)\tilde k(g) \in N_P A_P M_P K $
and $\lambda \in (\af_P)_\C^*$. 
Suppose you have shown that for 
all $K$-finite sections $f$ that $E(\lambda, \psi, f)$
can be meromorphically
continued to some region in the parameter space
$\Omega\subset (\af_P)_\C^*$. Then the same holds true for all
smooth sections $f$. 
\end{rem}

\section{Appendix A: Spherical principal series and the proof of 
Theorem \ref{thm=main}(i)} 

We first discuss  how the proof of Theorem \ref{thm=main}(i)  reduces to the case of 
spherical principal series. We use the notation from 
Section \ref{mpri}. 

\par First it is clear that it is sufficient 
to establish the result for {\it one} fixed set of generators $\xi_1, \ldots, \xi_k$ 
of $I(W)$.  Next let $\{0\}=W_0\subset W_1\subset\ldots \subset W_n=W$
be a Jordan-H\"older series of $W$. It induces an inclusion chain of $(\gf,
K)$-modules  
$$\{0\}=I(W_0)\subset I(W_1)\subset \ldots \subset I(W_n)=I(W)$$
with $I(W_{j+1})/I(W_j)\simeq I(W_{j+1}/ W_j)$. It follows that we can and
will assume
that $W$ is irreducible. In particular, the $P_{\rm min}$-representation $W$ 
factors to $P_{\rm min}/N\simeq M\times A$. Let us write $\sigma\times \chi$
for this $M\times A$-representation on $W$. Next there exists a
finite dimensional representation $F$ of $G$ with $N$-invariants $F^N$ such that 
$W\hookrightarrow  F^N\otimes\C_\chi\subset F\otimes \C_\chi$. 
Further $I(F\otimes\C_\chi) \simeq I(\C_\chi)\otimes F$ and thus we obtain an
embedding $I(W)\hookrightarrow I(\C_\chi)\otimes W$. 
In view of our discussion of tensoring with 
finite dimensional representations (see Subsection \ref{tensor})
matters reduce to $W=\C_\chi$.

\subsection{Spherical principal series representations}

In this section we introduce a Dirac-type sequence 
for spherical principal series representations (see Subsection \ref{dirac}
with Theorem \ref{th=1}). This allows us to establish lower
bounds for matrix-coefficients which are uniform in the $K$-types (cf. Corollary \ref{Cor1}). 
These lower bounds are essentially sharp, locally uniform in the 
representation parameter, and stronger than the more abstract
estimates in Theorem \ref{thm=lobo}. 

\par The lower bounds established give us a constructive method for finding 
Schwartz-functions representing a given smooth vector and as a side product 
a proof of Theorem \ref{thm=main}(i) for spherical principal 
series (see Subsection \ref{cs}).  

\medskip 
\par According to the Iwasawa decomposition 
$G=NAK$ we decompose elements $g\in G$ as 

$$g = \tilde n(g) \tilde a(g) \tilde k(g)$$
with $\tilde n(g)\in N$, $\tilde a(g)\in A$ and $\tilde k(g)\in K$. 
We recall $M=Z_K(A)$ and the minimal parabolic subgroup 
$P_{\rm min}=NAM$ of $G$. 

\par The Lie algebras of $A,N$ and $K$ shall be denoted 
by $\af$, $\nf$ and $\kf$.  Complexification of Lie-algebras 
are indicated with a $\C$-subscript, i.e. $\gf_\C$ is the complexification 
of $\gf$ etc. 
As usually we define $\rho\in \af^*$ by 
$\rho(Y):={1\over 2}\tr (\ad_\nf Y)$ for $Y\in \af$. 
\par The smooth spherical principal series with parameter 
$\lambda\in\af_\C^*$ is defined by 

\begin{align*} \H_\lambda^\infty:=\{ f\in C^\infty(G)\mid & (\forall nam\in 
P_{\rm min}, 
\forall g\in G)\\
& f(nam g) = a^{\rho+\lambda} f(g)\}\end{align*}
We note that $R$ defines a smooth representation of $G$ 
on $\H_\lambda^\infty$ which we denote henceforth by $\pi_\lambda$. 
The restriction map to $K$ defines a $K$-isomorphism: 

$$\Res_K: \H_\lambda^\infty \to  C^\infty(K\bs M), \ \ f\mapsto f|_K\, .$$
The resulting action of $G$ on $C^\infty(M\bs K)$ is given by 

$$[\pi_\lambda(g)f] (Mk)= f (M \tilde k(kg)) \tilde a(kg)^{\lambda+\rho}\, .$$ 
This action lifts to a continuous action on the Hilbert completion 
$\H_\lambda=L^2(M\bs K)$ of $C^\infty(M\bs K)$. We note that this 
representation is unitary provided that $\lambda\in i\af^*$. 

\par We denote by $V_\lambda$ the $K$-finite vectors 
of $\pi_\lambda$ and note that $V_\lambda=\C[M\bs K]$ as $K$-module. 
For later reference we record that the dual representation of 
$(\pi_\lambda, \H_\lambda)$ is isomorphic to 
$(\pi_{-\lambda}, \H_{-\lambda})$ via 
the $G$-equivariant pairing

\begin{equation}\label{pair1} (\cdot, \cdot ): \ \H_{-\lambda}\times \H_\lambda\to \C, \ \ (\xi, v):=
\int_{M\bs K} \xi (Mk) v (Mk)\  d(Mk)\, .\end{equation}
Here $G$-equivariance means that 
$$(\pi_{-\lambda}(g)\xi, v) = (\xi, \pi_\lambda(g^{-1}) v)$$
for all $g\in G$.  
 
\subsubsection{$K$-expansion of smooth vectors}

We recall $\hat K$, the set of equivalence classes of 
irreducible unitary representations of $K$. If $[\tau]\in \hat K$ we let $(\tau, U_\tau)$ be a 
representative. Further we write $\hat K_M$ for the subset of $M$-spherical 
equivalence classes, i.e. 
$$[\tau] \in \hat K_M \iff U_\tau^M:=\{ u\in U_\tau\mid  \tau(m)u=u\  \forall m\in M\}\neq \{0\}\,  .$$

\par Given a finite dimensional representation $(\tau, U_\tau)$ of $K$ we denote by 
$(\tau^*, U_\tau^*)$ its  
dual representation. With each $[\tau]\in \hat K_M $ comes the realization 
mapping 

$$r_\tau: U_\tau\otimes (U_\tau^*)^M\to L^2(M\bs K), 
\ \ u\otimes \eta\mapsto (Mk \mapsto \eta(\tau(k) u))\, .$$
Let us fix a $K$-invariant  inner product on $U_\tau$. This 
inner product induces a $K$-invariant inner product on $U_\tau^*$. 
We obtain an inner product on $U_\tau\otimes (U_\tau^*)^M$ 
which is independent of the chosen inner product on 
$U_\tau$. If we denote by $d(\tau)$ the dimension of $U_\tau$, then 
Schur-orthogonality implies that 
$$ {1\over d(\tau)} \|u\otimes \eta\|^2 = \|r_\tau(u\otimes \eta)
\|_{L^2(M\bs K)}^2\, .$$ 
Taking all realization maps together we arrive at a $K$-module 
isomorphism 
$$\C[M\bs K]=\sum_{\tau\in \hat K_M} 
U_\tau\otimes (U_\tau^*)^M\, .$$

\par Let us fix a maximal torus $\tf\subset \kf$ and a positive 
chamber $\Cc\subset i \tf^*$. We often identify $\tau$ with its 
highest weight in $\Cc$ and write $|\tau|$ for the
norm (with respect to the positive definite form $B$) of the highest weight. 
As $d(\tau)$ is polynomial in  $\tau$
we arrive at the following characterization of the smooth functions:

\begin{align*} C^\infty(M\bs K)=\Big\{ \sum_{\tau\in \hat K_M} c_\tau u_\tau\mid
& c_\tau\in \C , u_\tau \in  U_\tau\otimes (U_\tau^*)^M, \|u_\tau\|=1 \\
& (\forall N\in \N) \sum_{\tau\in \hat K_M} 
|c_\tau| (1+|\tau|)^N < \infty\Big\}\, .\end{align*}

Let us denote by $\delta_{Me}$ the point-evaluation of 
$C^\infty(M\bs K)$ at the base point $Me$. We decompose 
$\delta_{Me}$ into $K$-types: 

$$\delta_{Me}=\sum_{\tau\in \hat K_M} \delta_\tau$$
where 
$$\delta_\tau= d(\tau)\sum_{i=1}^{l(\tau)} u_i \otimes u_i^*$$
with $u_1, \ldots, u_{l(\tau)}$ any basis of 
$U_\tau^M$ and $u_1^*, \ldots,   u_{l(\tau)}^*$ its dual 
basis.
For $1\leq i,j \leq l(\tau)$ we set
$$\delta_\tau^{i,j}:= u_i \otimes u_j^*$$
and record that $\delta_\tau=d(\tau)\sum_{i=1}^{l(\tau)}\delta_\tau^{i,i}$. 
Note the following properties of $\delta_\tau$ and $\delta_\tau^{i,j}$: 

\begin{itemize}
\item $\|\delta_\tau^{i,i}\|_\infty=\delta_\tau^{i,i}(Me)= 1 $.
\item $\delta_\tau*\delta_\tau=\delta_\tau$. 
\item $\delta_\tau*f = f$ for all $f\in 
L^2(M\bs K)_\tau:= \im r_\tau$. 
\end{itemize}

\subsubsection{Non-compact model}

We have seen that the restriction map $\Res_K$ realizes $\H_\lambda^\infty$ 
as a function space on $M\bs K$. 
Another standard realization will be useful for us. Let us denote by 
$\oline N$ the opposite of $N$. 
As $NAM \oline N$ is open and dense in $G$ we obtain 
a faithful restriction mapping: 

$$\Res_{\oline N}: \H_\lambda^\infty\to C^\infty(\oline N), 
\ \ f\mapsto f|_{\oline N}\, .$$
Note that this map is not onto. The transfer of compact to 
non-compact model is given by 

\begin{align*} \Res_{\oline N}  \circ \Res_K^{-1}&: 
C^\infty(M\bs K)\to C^\infty (\oline N), \\
&f\mapsto F; \ F(\oline n):= \tilde a(\oline n )^{\lambda +\rho} f (\tilde k(\oline n ))
\end{align*} 

The transfer of the Hilbert space structure on $\H_\lambda= L^2(M\bs K)$ 
results in the $L^2$-space 
$L^2(\oline N,  \tilde a(\oline n)^{-2\Re \lambda} d\oline n)$ with $d\oline n$ an 
appropriately normalized  
Haar measure on $\oline N$. In the sequel we also write 
$\H_\lambda$ for $L^2(\oline N,  \tilde a(\oline n)^{-2\Re \lambda} d\oline n)$
in the understood context. The full action of $G$ in the non-compact model 
is not of relevance to us, however we will often use the $A$-action 
which is much more transparent in the non-compact picture: 

$$[\pi_\lambda(a) f ](\oline n) =  a^{\lambda +\rho}
f (a^{-1} \oline n a) $$
for all $a\in A$ and $f\in 
L^2(\oline N,  \tilde a(\oline n)^{-2\Re \lambda} d\oline n)$.

\subsubsection{ $K$-finite vectors with fast decay}

\par The fact that $\Res_K$ is an isomorphism follows from the 
geometric fact that $P_{\rm min}\bs G\simeq M\bs K$. Now $\oline N$ embeds into 
$P_{\rm min}\bs G= M\bs K$ as an open dense subset. In fact the complement is 
algebraic and we are going to describe it as the zero set 
of a $K$-finite functions $f$ on $M\bs K$. 
We will show that $f$ can be chosen such that $f$ restricted to $\oline N$ has 
polynomial decay of arbitrary fixed order.

\par Let $(\sigma, W)$ be a finite  dimensional faithful irreducible
representation of $G$. We assume that $W$ is $K$-spherical, i.e. $W$ 
admits a non-zero $K$-fixed vector, say $v_K$. 
It is known that $\sigma$ is $K$-spherical if and only if there is a
real line $L\subset W$ which is fixed under $\oline P_{\rm min}=MA\oline N$. Let 
$L=\R v_0$ and $\mu \in \af^*$ be such that 
$\sigma(a)v_0= a^\mu \cdot v_0$ for all $a\in A$, in other words:
$v_0$ is a lowest weight vector of $\sigma$ and $\mu$ is the 
corresponding lowest weight. 
\par Let now $\la\cdot, \cdot\ra$ be an inner product on $W$ which is 
$\theta$-covariant: if $g=k\exp(X) $ for $k\in K$ and $X\in \pf$
and $\theta(g):=k \exp(-X)$, then covariance means 
$$\la \sigma(g)v, w\ra = \la v , \sigma(\theta(g)^{-1}) w\ra$$
for all $v, w\in W$ and $g\in G$. Such an inner product is unique up to 
scalar by Schur's Lemma. Henceforth we request that $v_0$ is 
normalized and we fix $v_K$ by $\la v_0, v_K\ra = 1$. 
Consider on $G$ the function
$$f_\sigma(g):=\la \sigma(g)v_0, v_0\ra\, .$$ 
The restriction of $f_\sigma$ to $K$ is also denoted by 
$f_\sigma$. 

\par Let now $\oline n\in \oline N$ and write 
$\oline n = \tilde n(\oline n) \tilde a (\oline n) \tilde k(\oline n)$ according 
to the Iwasawa decomposition. Then 
$\tilde k(\oline n)= n^* \tilde a(\oline n)^{-1} \oline n $ for some $n^*\in  N$. 
Consequently 
$$f_\sigma(\tilde k(\oline n))= \tilde a(\oline n )^{-\mu}\, .$$
If $(\oline n_j)_{j}$ is a sequence in $\oline N$ such that 
$\tilde k(\oline n_j)$ converges to a point in $M\bs K - \tilde k(\oline N)
=:M\bs K -  \oline N$, then $\tilde a(\oline n_j)^{-\mu}\to 0$. Hence 
$$M\bs K -  \oline N\subset \{ Mk\in M\bs K\mid f_\sigma(k)=0\}\, .$$
As $f_\sigma$ is non-negative one obtains for all regular 
$\sigma$ that equality holds: 
$$M\bs K -  \oline N = \{ Mk\in M\bs K\mid f_\sigma(k)=0\}$$
(this reasoning is not new and goes back to Harish-Chandra). Let us 
fix such a $\sigma$ now. 
\par We claim that the mapping $\oline n \to f_\sigma(\oline n)$
is the inverse of a polynomial mapping, i.o.w. the map
$$\oline N\to \R, \ \  \oline n\mapsto \tilde a(\oline n)^\mu$$
is a polynomial map. But this follows from 
$$ \tilde a(\oline n )^\mu = \la \sigma(\oline n)v_K, v_0\ra$$
by means of our normalizations. 

\par In order to make estimates later on we introduce coordinates on 
$\oline N$. For that we first write $\oline \nf$ as 
semi-direct product of $\af$-root vectors: 

$$\oline \nf= \R X_1 \ltimes \left(\R X_2 \ltimes \left(\ldots
\ltimes \R X_n \right)\ldots\right)\, .$$ 
Accordingly we write elements of $\oline \nf$ as 
$X :=\sum_{j=1}^n x_j X_j$ with $x_i\in \R$. We note the 
following two facts:
\begin{itemize}
\item The map  \
$$\Phi: \oline \nf \to \oline N, \ \ X\mapsto \oline n(X):=
\exp(x_1 X_1)\cdot \ldots \cdot 
\exp(x_n X_n)$$
is a diffeomorphism. 
\item One can normalize the Haar measure $d\oline n $ of $\oline N$ in such a
way that: 
$$\Phi^*(d\oline n)= dx_1\cdot \ldots \cdot d x_n\, .$$ 
\end{itemize}
We introduce a norm on $\oline \nf$ by 
setting 
$$\|X\|^2:=\sum_{j=1}^n |x_j|^2\qquad (X\in \oline \nf)\, .$$ 

Finally we set 
$$f_\sigma(X):= f_\sigma(\tilde k(\oline n(X)))= 
\tilde a (\oline n(X))^{-\mu}$$
and summarize our discussion. 
\begin{lemma} Let $m>0$. Then there exists $C>0$ and  a finite dimensional 
$K$-spherical representation $(\sigma, W)$ of $G$ such that: 
\begin{enumerate}
\item $M\bs K - \oline N=\{ Mk\in M\bs K\mid f_\sigma (k)=0\}$. 
\item $|f_\sigma(X) |  \leq C\cdot (1 + \|X\|)^{-m}$
for all $X\in \oline \nf$. 
\end{enumerate}
\end{lemma}

\subsubsection{Dirac type sequences}\label{dirac}

Dirac sequences do not exist for Hilbert representations as they are 
features of an $L^1$-theory. However, rescaled they exist 
for the Hilbert representations we shall consider.

\par Recall our function $f_\sigma$ on $M\bs K$. We let 
$\xi=\xi_\sigma$ be the corresponding function transferred to $\oline N\simeq 
\oline \nf$ 
i.e. 
$$\xi(X):= \tilde a(\oline n (X))^{\rho +\lambda}
 f_\sigma(\tilde k(\oline n(X)))= \tilde a(\oline n (X))^{\rho +\lambda-\mu}
\, .$$
It is clear that $\xi$ is a $K$-finite vector for $\pi_{\lambda}$.

\par We recall that $\xi(X)$ satisfies 
the inequality  
$$|\xi(X)|\leq C\cdot (1+\|X\|)^{-m}$$
where we can choose $m$ as large as 
we wish (provided $\sigma$ is sufficiently regular and large). 
Record the normalization $\xi(0)=1$.

We will chose $m$ at least that large that $\xi$ becomes 
integrable and write 
$\|\xi\|_1$ for the corresponding 
$L^1(\oline N)$-norm. 

\par The operators $\pi_\lambda(a)$ can be understood as 
scaling operators in the non-compact picture. For our purpose the 
scaling in one direction of $A$ will be sufficient. To make this
precise we fix an element 
$Y\in \af$ such that $\alpha(Y)\geq 1$ for all roots 
$\alpha\in \Sigma(\af, \nf)$. For $t>0$ we put 
$$a_t:=\exp((\log t) Y)\, .$$ 
Note that for $\eta\in \af_\C^*$ one has 
$$a_t^\eta=t^{\eta(Y)}\, .$$
In the sequel we will often abbreviate and simply write 
$t^\eta$ for $t^{\eta(Y)}$.

\par In order to explain the idea of this section let us assume for a moment 
that $\lambda$ is real.  Then $\xi$ is a positive function and   

$$\left( {a_t^{\rho-\lambda}\over \|\xi\|_1}\cdot 
\pi_\lambda(a_t)\xi\right)_{t>0}$$
forms a Dirac sequence for $t\to \infty$ (If $\lambda$ is not real, then 
$\xi$ is oscillating and we have to be slightly more careful). 

In the compact picture this means 
$$\lim_{t\to\infty} {a_t^{\rho-\lambda}\over \|\xi\|_1}\cdot 
\pi_\lambda(a_t)f_\sigma =\delta_{Me}= \sum_{\tau\in\hat K_M} \delta_\tau\, .$$
It is our goal to understand this limit in the $K$-types: How large do we have to 
choose $t$ in dependence of $\tau$ such that the $\tau$-isotypical part 
of ${a_t^{\rho-\lambda}\over \|\xi\|_1}\cdot 
\pi_\lambda(a_t)f_\sigma$ approximates  $\delta_\tau$ well. It turns out that 
$t$ can be chosen polynomially in $\tau$.
If we denote by $D_\tau$ the transfer of the character $\delta_\tau$ to the 
non-compact model, the precise 
statement is as follows.

\begin{thm} Let $\lambda \in \af_\C^*$ and $N>0$. 
Then there exists a choice of $\sigma$ and hence of $\xi=\xi_\sigma\in
V_\lambda$, constants $c>2$, $C>0$ 
 such that for all $\tau\in \hat K_M$
one has 
$$[\pi_\lambda(a_{t(\tau)})\xi]_\tau= a_{t(\tau)}^{-\rho+\lambda} \cdot I_\xi\cdot 
 D_\tau + R_\tau$$
where $t(\tau):= (1+|\tau|)^c$, 
$$I_\xi:=\int_{\oline N} \xi(\oline n) \ d\oline n\neq 0$$
and remainder $R_\tau \in \H_\lambda[\tau]$ satisfying 
$$ {\|R(\tau)\|\over |a_{t(\tau)}^{-\rho+\lambda}|}\leq {C\over  (1+|\tau|)^N}\, .$$
\end{thm}

\begin{proof} Recall the $M$-fixed  functions 
$\delta_\tau^{i,j} \in L^2(M\bs K)_\tau$, 
$1\leq i,j \leq l(\tau)$
for $\tau \in \hat K_M$. In the sequel we abbreviate 
and set $d:=d(\tau)$, $l:=l(\tau)$.

Let $D_\tau^{i,j}(\oline n)= \tilde a(\oline n)^{\rho +\lambda} \delta_\tau^{i,j} 
(\tilde k(\oline n))$ the transfer of $\delta_\tau^{i,j}$ to the non-compact 
model. 
We also set $D_\tau^{i,j}(X):=D_\tau^{i,j}(\oline n(X))$ 
for $X\in \oline \nf$. 
Let us note that $|D_\tau^{i,j}(0)|=\delta_{ij}$. 

As $\pi_\lambda(a) \xi$ is $M$-fixed for all $a\in A$ we conclude that 

$$[\pi_\lambda(a_t)\xi]_\tau=\sum_{i,j=1}^l b_{i,j} (t) \cdot 
d \cdot D_\tau^{ij}\, .$$
If $\la\cdot, \cdot\ra$ denotes the Hermitian bracket on
$\H_\lambda=L^2(\oline N,  \tilde a(\oline n)^{-2\Re \lambda} d\oline n)$, 
then the coefficients $b_{i,j}(t)$ are obtained by 
the integrals

$$ b_{i,j}(t) =\la \pi_\lambda(a_t)\xi, D_\tau^{i,j}\ra=
\int_{\oline\nf} (\pi_\lambda(a_t)\xi)(X) 
\cdot \oline {D_\tau^{i,j}(X)} \cdot \tilde a(\oline n(X))^{-2\Re\lambda}\
dX\, , 
$$
where we used the notation 
$$dX:=dx_1\cdot \ldots\cdot dx_n\, $$ 
for $X=\sum_{j=1}^n x_j X_j$.
\par Fix $1\geq t_0>0$ and set $t=t_0^{-2}$.

We split the integrals for $ b_{i,j}(t)$ into two parts 
$ b_{i,j}(t)= b_{i,j}^1(t)+ b_{i,j}^2(t) $ with

$$b_{i,j}^1( t):=\int_{\{ \|X\|\geq t_0 \}} (\pi_\lambda(a_t)\xi)(X) 
\cdot \oline {D_\tau(X)} \cdot \tilde a(\oline n(X))^{-2\Re\lambda}\ dX\,  . $$
In our first step of the proof we wish to estimate $b_{i,j}^1( t)$. 
For that let $C, q_1>0$ be such that 
$$\tilde a(\oline n(X))^{-2\Re\lambda}\leq C \cdot (1+\|X\|)^{q_1}\, .  $$
Likewise, by the definition of $D_\tau^{i,j}$ we obtain constants
$C, q_2>0$ which only depend on $\Re \lambda$ and 
such that 
$$|D_\tau^{i,j} (X)|\leq C \cdot(1+\|X\|)^{q_2}$$
for all $\tau$ and $1\leq i,j\leq l$. Set $q:=q_1+q_2$.

From the inequalities just stated 
we arrive at:

$$|b_{i,j}^1(t) |\leq  C \cdot 
 t^{\Re \lambda +\rho}  
\int_{\{ \|X\|\geq  t_0\}} 
|\xi(\Ad(a_t)^{-1} X)|\cdot (1+\|X\|)^q  \ dX\, .$$
As $|\xi(X)|\leq C \cdot (1+\|X\|)^{-m}$ for some constants 
$C, m>0$ we thus get that 

$$|b_1^i(\tau,t)|\leq  C \cdot 
 t^{\Re \lambda +\rho}  
\int_{\{ \|X\|\geq  t_0\}} 
{(1+\|X\|)^q\over (1+ \|\Ad(a_t)^{-1}X\|)^m}  \ dX\, .$$
By the definition of $a_t$ we get that $\|\Ad(a_t)^{-1} X\|\geq t\|X\|$
and hence 
$$|b_{i,j}^1(t)| \leq  C \cdot 
 t^{\Re \lambda +\rho}  
\int_{\{ \|X\|\geq  t_0\}} 
{(1+\|X\|)^q\over (1+ t \|X\|)^m}  \ dX\, .$$
We continue this estimate by employing 
polar coordinates for $X\in \oline \nf$: 

\begin{align*} 
|b_{i,j}^1(t)|&  \leq C \cdot  t^{\Re \lambda+\rho}
\int_{t_0}^\infty {r^n  (1+r)^q \over  (1+ tr)^m} \ {dr\over r}\\
& =   C\cdot t_0^{n-2({\Re \lambda}+\rho)}  
\int_1^\infty {r^n  (1+t_0 r)^q \over  (1+ tt_0 r)^m} \ {dr\over r}\\
& =   C\cdot t_0^{n-2({\Re \lambda} +\rho)}
\int_1^\infty {r^n (1+t_0 r)^q  \over  (1+ t_0^{-1} r )^m} \ {dr\over r}\\
&=  C\cdot t_0^{n-2 {(\Re \lambda+\rho)} +m}
\int_1^\infty {r^n  (1+t_0 r)^q \over  (t_0 + r )^m} \ {dr\over r}\\
&\leq  C \cdot t_0^{n-2({\Re \lambda}+\rho) +m}
\int_1^\infty  r^{n+q-m} \ {dr\over r}\, .
\end{align*}

Henceforth we request that $m> n+q+1$. Thus 
for every $m'>0$ there exist a choice of $\xi$ and 
a  constant $C>0$ such that
\begin{equation} \label{rest1} |b_{i,j}^1(t)|\leq C \cdot t^{-m'}\, .\end{equation}

\par Next we choose $t$ in relationship to $|\tau|$. 
Basic finite 
dimensional  representation theory yields that 
in a fixed  compact neighborhood of $X=0$ 
the gradient of $D_\tau^{i,j}$ is bounded by $C\cdot (1+|\tau|)$
for a constant $C$ independent of $\tau$. Let $\gamma>1$. 
Then for $\|X\|\leq  (1+|\tau|)^{-\gamma}$ the mean value theorem 
yields the following estimate 

\begin{equation}\label{De} 
|D_\tau^{i,j}(X)- D_\tau^{i,j}(0)|\leq C\cdot (1+|\tau|)^{-\gamma +1}
\end{equation} 
This brings us to our choice of $t$, namely 

$$t=t(\tau):= (1+|\tau|)^{2\gamma}\, .$$

Recall the definition of 

$$I_\xi =\int_{\oline\nf} \xi(X)\  dX$$

Here we might face the obstacle that $I_\xi$ might be zero. 
However as $\xi(X)=\tilde a(n(X))^{\rho+\lambda-\mu}$ it follows 
$I\neq 0$ provided $\mu$ is large enough. 
So for any $m'$ we find such a non-zero $I_\xi$. 
\par In the following computation we will use 
the simple identity: 
$$\int_{\oline \nf} \pi_\lambda(a_t) f(X)\ dX = t^{\lambda-\rho} 
\int_{\oline \nf} f(X) \ dX$$
for all integrable functions $f$. 
Now if $i\neq j$, then $D_\tau^{i,j}(0)=0$ and we obtain from 
(\ref{rest1}) and (\ref{De})  that

\begin{align*} b_{i,j}(t)
&= \int_{\{\|X\|\leq t_0\}} 
(\pi_\lambda(a_t)\xi)(X)\cdot\oline{ D_\tau^{i,j}(X)} \ \tilde
a(\oline n(X))^{-2\Re \lambda} dX \\
&\quad + O\left({1\over (1+|\tau|)^{2\gamma m'}}\right)\\
& =  \int_{\{\|X\|\leq t_0\}}  (\pi_\lambda(a_t)\xi)(X)
\oline{\left( D_\tau^{i,j}(X)-D_\tau^{i,j}(0)\right)} \ \tilde
a(\oline n(X))^{-2\Re \lambda} \ dX\\
&\quad +O\left({1\over (1+|\tau|)^{2\gamma m'}}\right)\\ 
&=   t^{\lambda-\rho}\cdot \|\xi\|_1 \cdot 
O\left({1\over (1+|\tau|)^{\gamma-1}}\right)
+O\left({1\over (1+|\tau|)^{2\gamma m'}}\right)\, .\end{align*}

For $i=j$ we have $D_\tau^{i,i}(0)=1$ and we obtain in a similar fashion 
that

\begin{align*}  b_{i,i}(t) &= \int_{\{\|X\|\leq t_0\}}  (\pi_\lambda(a_t)\xi)(X)
\cdot \tilde a(\oline n(X))^{-2\Re \lambda} \ dX\\
& + t^{\lambda-\rho}\cdot \|\xi\|_1 \cdot 
O\left({1\over (1+|\tau|)^{\gamma-1}}\right)
+O\left({1\over (1+|\tau|)^{2\gamma m'}}\right) \\
 &= t^{\lambda -\rho} \cdot I_\xi + O\left({1\over (1+|\tau|)^{\gamma-1}}\right)
+O\left({1\over (1+|\tau|)^{2\gamma m'}}\right) \, .\\
\end{align*}
If we choose $c:=2\gamma$ and $\gamma-1 = N$ and $m'$ large enough, 
the assertion of the theorem follows.   
\end{proof}

The proof of the theorem shows that the approximation can be made 
uniformly on any compact subset $Q\subset \af_\C^*$. 
We further observe that $a_{t(\tau)}$ is bounded from above and below 
by powers of $1+|\tau|$. If switch to the compact models $\H_\lambda=L^2(M\bs
K)$ and denote $f_\sigma$ also by $\xi$, 
then an alternative 
version of the theorem is as follows:

\begin{thm}\label{th=1} Let $Q\subset \af_\C^*$ be a compact subset and $N>0$. 
Then there exists $\xi\in \C[M\bs K]$ and constants $c_1,c_2>0$ 
such that for all $\tau\in \hat K_M$, $\lambda\in Q$,   there exists 
$a_\tau\in A$, independent 
of $\lambda$,  with $\|a_\tau\|\leq (1+|\tau|)^{c_1}$ 
and numbers $b(\lambda,\tau)\in \C$  
such that 
$$\|[\pi_\lambda(a_\tau) \xi]_\tau  - b(\lambda,\tau) \delta_\tau\|
\leq {1\over (|\tau|+1)^{N+c_2}}$$ 
and 
$$|b(\lambda,\tau)|\geq {1\over (1+|\tau|)^{c_2}}\, .$$
Here $\|\cdot\|$ refers to the norm in $L^2(M\bs K)$. 
\end{thm}

Finally we deduce the following lower  bound for matrix coefficients.
Recall the non-degenerate complex bilinear 
$G$-equivariant pairing $(\cdot, \cdot)$ between 
$\H_\lambda$ and $\H_{-\lambda}$.

\begin{cor} \label{Cor1} Let $Q\subset \af_\C^*$ be a compact subset. 
Then there exists $\xi \in \C[M\bs K]$, 
constants $c_1, c_2, c_3>0$ such that 
$$\sup_{g\in G \atop \|g\|\leq (1+|\tau|)^{c_1}}
|(\pi_\lambda(g)\xi, v)| \geq c_2 {1\over (1+|\tau|)^{c_3}} \|v\|$$
for all $\lambda \in Q$, 
$\tau\in \hat K_M$ and $v\in V_{-\lambda}[\tau]$. Here $\|v\|$ refers 
to the norm on $\H_{-\lambda}= L^2(M\bs K)$.
In particular there exist a $s\in \R$ 
such that 
$$\sup_{g\in G \atop \|g\|\leq (1+|\tau|)^{c_1}}
|(\pi_\lambda(g)\xi, v)\geq c_2 \|v\|_{s, K}$$
for all  $\lambda \in Q$, $\tau\in \hat K_M$ and $v\in V_{-\lambda}[\tau]$. 
\end{cor}

Thus Theorem 7.1 in conjunction with 
the above Corollary yields the Casselman-Wallach Theorem for spherical 
principal series: 

\begin{cor}\label{c=unique} Let $\lambda\in \af_\C^*$ and $V_\lambda$ the  
Harish-Chandra module of the corresponding 
spherical principal series. Then $V_\lambda$ admits a unique smooth 
Fr\'echet globalization. 
\end{cor}

\subsubsection{Constructions in the Schwartz algebra}\label{cs}

Let us fix  a relatively compact open 
neighborhood $Q\subset \af_\C^*$. We choose the $K$-finite 
element $\xi\in \C[M\bs K] $ such that the conclusion 
of Theorem \ref{th=1} is satisfied. 

\begin{lem}\label{l=123} Let $U$ be an $\Ad(K)$-invariant neighborhood of 
$\bf 1$ in $G$ and $\F(U)$ the space of $\Ad(K)$-invariant 
test functions supported in $U$. Then there exists a holomorphic map 

$$Q\to \F(U), \ \ \lambda\mapsto h_\lambda$$
such that $\Pi_\lambda(h_\lambda)\xi=\xi$. 
\end{lem}

\begin{proof} Let $V_\xi\subset \C[M\bs K]$ be the $K$-module generated 
by $\xi$. Let $n:=\dim V_\xi$. 
Let $U_0$ be a $\Ad(K)$-invariant neighborhood of ${\bf 1} \in G$ such that 
$U_0^n \subset U$.

\par Note that any $h\in \F(U_0)$ induces operators 

$$ T(\lambda):= \Pi_\lambda(h)|_{V_\xi} \in \End (V_\xi)\, .$$ 
The compactness of $Q$ allows us to employ uniform Dirac-approximation: 
we can choose $h$ such that 

$$Q\to \Gl(V_\xi), \ \ \lambda\mapsto T(\lambda)$$
is defined and holomorphic. 
Let $n:=\dim V_\xi$. By Cayley-Hamilton $T(\lambda)$ is a zero of its
characteristic polynomial and hence 
$$  \id_{V_\xi}= {1\over \det T(\lambda)}\sum_{j=1}^n 
c_j(\lambda) T(\lambda)^j$$
with $c_j(\lambda)$ holomorphic. Set now 

$$h_\lambda:={1\over \det T(\lambda)}\sum _{j=1}^n c_j(\lambda) 
\underbrace{h_\lambda * \ldots * h_\lambda}_{\hbox{j-times}}\, .$$
Then $Q\ni\lambda \mapsto h_\lambda\in \F(U)$ is holomorphic 
and $\Pi_\lambda(h_\lambda)\xi =\xi$. \end{proof}

\par For a compactly supported measure $\nu$ on $G$ 
and $f\in \S(G)$ we define $\nu*f\in \S(G)$ by 

$$\nu*f(g)=\int_G f(x^{-1}g) \ d\nu(x)\,  .$$
For an element $g\in G$ we denote by $\delta_g$ the 
Dirac delta-distribution at $g$. Further we view 
$\delta_\tau$ as a compactly supported measure on $G$ 
via the correspondence
$\delta_\tau\leftrightarrow \delta_\tau(k)\ dk$.

\par For  each $\tau\in \hat K_M$ we define 
$h_{\lambda, \tau} \in \S(G)$ by 

\begin{equation} \label{he} h_{\lambda, \tau}:=
\delta_\tau * \delta_{a_{t(\tau)}}* h_\lambda\, . \end{equation}

Call a sequence $(c_\tau)_{\tau \in \hat K_M}$ 
rapidly decreasing if 

$$ \sup_{\tau} |c_\tau| (1+|\tau|)^R<\infty$$
for all $R>0$. 

\begin{lem} \label{rapid} Let $(c_\tau)_\tau$ be a rapidly 
decreasing sequence $(c_\tau)_\tau$ and $h_{\lambda, \tau}$ defined as in 
(\ref{he}). Then 
$$H_\lambda:=\sum_{\tau\in \hat K_M} c_\tau \cdot h_{\lambda, \tau}$$ 
is in $\S(G)$ and the assignment $Q\ni \lambda\to H_\lambda \in \S(G)$
is holomorphic. 
\end{lem}

\begin{proof} Fix $\lambda\in Q$. For simplicity set $H=H_\lambda$, 
$h_{\lambda, \tau}=h_\tau$. 

It is clear that the convergence of $H$  
is uniform on compacta and hence $H\in C(G)$.  
For $u\in \U(\gf)$ we record 
$$R_u (h_\tau)=\delta_\tau * \delta_{a_{t(\tau)}}* R_u(h)$$
and as a result $H\in C^\infty(G)$. So we do not have to 
worry about right derivatives.
To show that $H\in \S(G)$ we employ Remark \ref{rem=r=s}: 
it remains to show that $H\in \cR(G)$, i.e.

\begin{equation} \label{hsup} \sup_{g\in g} \|g\|^r \cdot |H(g)|<\infty
\end{equation}
for all $r>0$. 
Fix $r>0$. Write $g=k_1 a k_2$ for some $a\in A$, $k_1, k_2\in K$. 
Then 

$$\|g\|^r |h_\tau(g)|\leq  \|a\|^r \cdot \sup_{k, k'\in K}
|h(a_t^{-1} k a k')|\, .$$
Let $Q\subset A$ be a compact set with 
$\log Q$ convex and $\W$-invariant and such that 
$\supp h \subset K Q K$. We have to determine those 
$a\in A$ with 

\begin{equation} \label{nemp} a_t^{-1}Ka\cap KQK\neq \emptyset\, .
\end{equation}
Define $Q_t\subset A$ through   
$\log Q_t$ being the convex hull of $\W(\log a_t +\log Q)$.
Then 
(\ref{nemp}) implies that 
$$a\in Q_t\, .$$
But this means that $\|a\|<<|\tau|^c$ for some $c>0$, independent 
of $\tau$. Hence (\ref{hsup}) is verified and $H$ is 
indeed in $\S(G)$. 

\par Finally the fact that the assignment $\lambda \mapsto H_\lambda$ 
is holomorphic follows from the previous Lemma. 
\end{proof}

\begin{thm} \label{th=2} Let $Q\subset \af_\C^*$ be a compact 
subset. Then there exist a continuous map 
$$Q\times C^\infty(M\bs K) \to \S(G), \ \ (\lambda, v)\mapsto 
f(\lambda, v)$$
which is holomorphic in the first variable, linear in the second 
and such that 
$$\Pi_\lambda(f(\lambda, v))\xi = v\, .$$ 
In particular, $\Pi_\lambda(\S(G))V_\lambda=\H_\lambda^\infty$ for all 
$\lambda\in \af_\C^*$.  
\end{thm}

\begin{proof} Let $v\in \H_\lambda^\infty$. Then 
$v=\sum_{\tau} c_\tau v_\tau$ with $v_\tau$ normalized and 
$(c_\tau)_\tau$ rapidly decreasing. As $\S(G)$ 
is stable under left convolution 
with $C^{-\infty}(K)$ we readily reduce to the case 
where $v_\tau =  {1\over \sqrt{d(\tau)}} \delta_\tau$. 

\par In order to explain the idea of the proof let us first treat 
the case  where the Harish-Chandra module is a multiplicity free 
$K$-module. This is for instance satisfied when $G=\Sl(2,\R)$. 
  
\par Recall the numbers $b(\lambda, \tau)$ from Theorem \ref{th=1} 
and define 
$$H_\lambda:= \sum_\tau{} {c_{\tau} \over \sqrt{d(\tau)} \cdot b(\lambda,\tau)} h_{\lambda, \tau}\, .$$
It follows from Theorem \ref{th=1} and  the Lemma above that 
$Q\ni \lambda\to H_\lambda\in \S(G)$ is defined and holomorphic. 
By multiplicity one we get that 
$$\Pi_\lambda(H_\lambda)\xi  =\sum_\tau  c_\tau v_\tau\, .$$
and the assertion follows for the multiplicity free case.

\par Let us  move to the general case. For that we employ the 
more general approximation in Theorem \ref{th=1} and set 

$$H_\lambda' =\sum_{\tau\in \hat K_M} 
{c_{\tau} \over \sqrt{d(\tau)}\cdot b(\lambda, \tau)} 
h_{\lambda, \tau}\, .$$

Then $$\Pi_\lambda(H_\lambda')\xi = \sum_{\tau\in \hat K_M}  c_\tau v_\tau
+ R$$
where, given $k>0$, we can assume that $\|R_\tau\|\leq |c_\tau| \cdot 
(|\tau|+1)^{-k}$ for all $\tau$ (choose $N$ in 
Theorem \ref{th=1} big enough). Finally we remove the remainder $R_\tau$ 
by left convolution with $C^{-\infty}(K)$ (use the Neumann series $(\id +R)^{-1}$). 
\end{proof}

\section{Appendix B: On the meromorphic extension of certain
distributions on $G\times G$}\label{app}

Let $X$ and $Y$ be real smooth affine varieties. We may view $X$, resp. $Y$, 
as Zariski closed subsets in $\R^n$, resp. $\R^m$. 
Let $\S'(X\times Y)$ the space of tempered distributions
on $X\times Y$. Further let $p: Y\to \R^+$ be a positive polynomial 
function such that $p(y)> \|y\| \quad (y\in Y)$ for some norm $\|\cdot\|$ 
on $\R^m$. 
We consider the canonical projections $\pi: X\times Y \to X$
and $\rho: X\times Y \to Y$. 
Let $\E\in \S'(X\times Y)$. Since $\E$ is tempered, there exists an $r_0>0$ such that for all 
$\lambda\in \C$ with $\Re \lambda >r_0$ the prescription 
$$\I(\lambda):=\pi_*(p^{-\lambda} \E)$$
defines a distribution on $X$. Furthermore, if $D'(X)$ denotes the 
topological vector space of distribution on $X$, then the assignment
$$\{z\in \C \mid \Re z>r_0\} \ni \lambda \to \I(\lambda)\in D'(X)\, .$$ 
is meromorphic. Note the formula 
$$\I(\lambda)(\phi) =\E ( p^{-\lambda}\circ \rho \cdot  \phi\circ \pi) \qquad 
(\phi\in C_c^\infty(X))\, .$$ 

\par Let  $\D=\D(X\times Y)$ be the ring of differential operators 
on $X\times Y$ with polynomial coefficients. 

Suppose now that $\E$ is holonomic, that is the $\D$-module 
generated by $\E$ in $\S'(X\times Y)$ is holonomic. 
A slight modification of the main result in \cite{B} then yields 
differential operators $d_1(\lambda), \ldots , d_k(\lambda) \in \D(X)$, 
polynomially depending on $\lambda$, as well as a polynomial 
function $b(\lambda)$  such that 

$$b(\lambda) \I(\lambda+k)= d_1(\lambda) \I(\lambda + k-1) +
\ldots + d_k(\lambda) \I(\lambda)\, .$$
In particular, $\I(\lambda)$ admits a meromorphic continuation 
as a distribution on $X$. 
 
\par In the sequel we will use this result for $X=G\times G$ and $Y=G$. 
Let $V$ be a Harish-Chandra module of a discrete series representations. 
Let $I=\Ind_{P_{\rm min}}^G (\sigma_1)$ be a minimal principal series
representation with respect to an irreducible representation $\sigma_1$ 
of $P_{\rm min}/ N= M\times A$ such that $V\hookrightarrow I$. 
Likewise we choose a minimal principal series $J:=\Ind_{P_{\rm min}}^G
(\sigma_2)$ such that $\tilde V\hookrightarrow J$. Write $W_1$, $W_2$ for the 
representation module for $\sigma_1$, resp. $\sigma_2$.  
Let $\nu_{1,2}\in W_{1.2}$ be fixed non-zero elements. Let us consider 
the continuous surjections:

$$\Phi: C_c^\infty(G) \to I^\infty, \ \ \phi\mapsto 
\Big(g\mapsto \int_{P_{\rm min}} \phi(pg) \sigma_1(p)^{-1} \nu_1 \ dp \Big)$$
and 
$$\Psi: C_c^\infty(G) \to J^\infty, \ \ \psi\mapsto 
\Big(g\mapsto \int_{P_{\rm min}} \psi(pg) \sigma_2(p)^{-1} \nu_2 \ dp \Big)\, .$$

Further we let $\xi \in \tilde I$ and $\eta\in \tilde J$ be such that the 
maps 

$$ V \to C^\infty(G), \ \ v\mapsto m_{\xi, v}$$
$$\tilde V\to C^\infty(G), \ \ w \mapsto m_{\eta, w}$$
are injective.

Now for $\lambda\in \C$ with $\Re \lambda$ sufficiently big, the prescription 

$$\I(\lambda)(\phi, \psi):= \int_G m_{\xi, \Phi(\phi)}(g) 
m_{\eta, \Psi(\psi)}(g) \|g\|^{-\lambda} \ dg \qquad 
(\phi, \psi\in C_c^\infty(G)) $$
defines a distribution on $G\times G$. We claim that $\I(\lambda)$ admits 
a meromorphic continuation to the complex plane. 
In fact
$$\E(g, h_1, h_2):= \xi(h_1 g^{-1})(\nu_1) \cdot \eta(h_2g^{-1})(\nu_2)\, $$
defines a moderately growing function on $G\times G \times G$, hence 
a tempered distribution.  With respect to the first variable 
projection $\pi: G\times G \times G \to G$ 
we readily verify the identity 

$$\I(\lambda)=\pi_*(\|\cdot\|^{-\lambda}\E)\, .$$  
The fact that $\E$ is holonomic implies the claim.

\end{document}